\documentclass[11pt, oneside]{amsart}
\usepackage{amsmath}
\usepackage{amsthm}
\usepackage{amssymb} 
\usepackage[mathscr]{euscript}
\usepackage{amsmath}
\usepackage{amsthm}
\usepackage{amssymb}
\usepackage{fancyhdr}
\usepackage{enumitem}
\usepackage{amscd}
\usepackage[all]{xy}
\usepackage{graphicx}
\usepackage{color}
\usepackage{hyperref}
\usepackage{tikz}
\usepackage{amsrefs,verbatim}

\usepackage[all]{xy}
\usepackage{graphicx}
\usepackage{subcaption} 
\usepackage{cancel} 
\usetikzlibrary{backgrounds}

\theoremstyle{plain}
\newtheorem{lemma}{Lemma}[section]

\newtheorem{proposition}[lemma]{Proposition}

\newtheorem{theorem}[lemma]{Theorem}
\newtheorem{Theorem}{Theorem}
\newtheorem{Corollary}[Theorem]{Corollary}

\newtheorem{corollary}[lemma]{Corollary}

\numberwithin{equation}{section}
\newtheorem{manualtheoreminner}{Theorem}
\newenvironment{manualtheorem}[1]{%
  \renewcommand\themanualtheoreminner{#1}%
  \manualtheoreminner
}{\endmanualtheoreminner}

\theoremstyle{remark}

\theoremstyle{definition}
\newtheorem{remark}[lemma]{Remark}
\newtheorem{example}[lemma]{Example}
\newtheorem{definition}[lemma]{Definition}

 \DeclareMathOperator{\C}{\mathcal{C}}
  \DeclareMathOperator{\Int}{Int}
 \DeclareMathOperator{\mb}{\mathsf{b}}
  \DeclareMathOperator{\M}{\mathsf{M}}
  \DeclareMathOperator{\Ric}{Ric}
   \DeclareMathOperator{\Det}{det}
 \DeclareMathOperator{\Div}{div}
  \DeclareMathOperator{\tr}{tr}
    \DeclareMathOperator{\adj}{adj}
    \DeclareMathOperator{\range}{Range}

\newcommand{\nocontentsline}[3]{}
\newcommand{\tocless}[2]{\bgroup\let\addcontentsline=\nocontentsline#1{#2}\egroup}

\usepackage[tmargin=1in,bmargin=1in,lmargin=1in,rmargin=1in]{geometry}
\begin{document}
\title[Scalar curvature deformation and mass rigidity]{Scalar curvature deformation and mass rigidity for ALH manifolds with boundary}
\author{Lan-Hsuan Huang}
\address{Department of Mathematics, University of Connecticut, Storrs, CT 06269, USA}
\email{lan-hsuan.huang@uconn.edu}
\author{Hyun Chul Jang}
\address{Department of Mathematics, University of Miami, Coral Gables, FL 33146, USA}
\email{h.c.jang@math.miami.edu}
\thanks{Both authors were partially supported by the NSF CAREER Award DMS-1452477. The first author was also partially supported by NSF DMS-2005588.}

\begin{abstract}
We study scalar curvature deformation for asymptotically locally hyperbolic (ALH) manifolds with nonempty compact boundary. We show that the scalar curvature map is locally surjective among either (1) the space of metrics that coincide exponentially toward the boundary, or (2) the space of metrics with arbitrarily prescribed nearby Bartnik boundary data. Using those results, we characterize the ALH manifolds that minimize the Wang-Chru\'sciel-Herzlich mass integrals in great generality and establish the rigidity of the positive mass theorems.  
\end{abstract}

\maketitle
{\small
\tableofcontents
\addtocontents{toc}{\setcounter{tocdepth}{2}} }
\section{Introduction}

Let $n\ge 3$ and $(N, h)$ be any closed $(n-1)$-dimensional Riemannian manifold. For $k\in \{ -1, 0 , 1\}$, we define a \emph{reference manifold} $\M = (r_k, \infty)\times N$ and $\mb = \frac{1}{r^2+k} dr^2 + r^2 h$ where $r_k = 0$ if $k=0, 1$, and $r_k=1$ if $k=-1$. A reference manifold is conformally compact with the conformal infinity $(N, h)$ and has an \emph{asymptotically locally hyperbolic} (or \emph{ALH} for short) end where the sectional curvature goes to $-1$ as $r\to \infty$.  Our reference manifolds include
\begin{itemize}
\item Locally hyperbolic manifolds, i.e., the sectional curvature of $\mb$ is identically $-1$, by letting $(N, h)$ have  constant sectional curvature $k$.
\item  Birmingham-Kottler manifolds which are Poincar\'e-Einstein, i.e. $\Ric_{\mb} = -(n-1) \mb$, by letting  $(N, h)$ have constant Ricci curvature $\Ric_h = k(n-2)h$). 
\end{itemize}
This paper concerns a general class of ALH manifolds (not necessarily conformally compact)  that are asymptotic to a reference manifold. By default (Definition~\ref{def:locallyhyperbolic}), our ALH manifold has  one end. While the results here can be modified to accommodate multiple ALH ends, we will not pursue in this paper.


In prior joint work with D. Martin, we have established the results on scalar curvature deformation for asymptotically hyperbolic manifolds \emph{without boundary} and characterized the hyperbolic space as the equality case in the positive mass theorem for asymptotically hyperbolic manifolds~\cite{Huang:2020tm}. To extend those results to more general ALH manifolds, it is necessary to take into account of nonempty (compact) boundary because a reference manifold can have a ``cuspidal'' end as $r\to 0$ when $k=0$ or contain a  minimal hypersurface ``neck''  when $k=-1$.  In order to characterize those reference manifolds, we will consider them as ALH manifolds with nonempty boundary by chopping off the cuspidal end or cutting along the minimal hypersurface boundary.

A fundamental problem in the study of scalar curvature is that given a smooth manifold $M$ and a scalar function $\mathsf{R}$, whether one can find a Riemannian metric $\gamma$ whose scalar curvature realizes $\mathsf{R}$; namely, $R_\gamma =\mathsf{R}$. We will address this problem for small deformation of an ALH manifold. That is, given an ALH manifold $(M, g)$ and $\mathsf{R}$ sufficiently close to $R_g$, we shall find an ALH metric $\gamma$, close to $g$, such that $R_\gamma = \mathsf{R}$. When the manifold has nonempty boundary $\Sigma$, one can further impose boundary conditions for~$\gamma$. We will mainly consider two types of boundary conditions for~$\gamma$:
\begin{itemize}
\item The metric $\gamma$ converges to $g$ toward $\Sigma$ exponentially, as well as their derivatives up to (at least) the second order. 

\item For sufficiently small $(\tau, \phi)$, $\gamma$ satisfies $(\gamma^\intercal, H_\gamma) = (g^\intercal, H_g)+(\tau, \phi)$, where $^\intercal$ denotes the restriction on the tangent bundle of $\Sigma$ and $H_g$ is the mean curvature. 
\end{itemize}
The pair $ (\gamma^\intercal, H_\gamma)$ is often called the \emph{Bartnik boundary data}, so the second boundary condition is to prescribe the Bartnik boundary data of $\gamma$. 
 
 \vspace{6pt}
\noindent{\bf Sign convention for the mean curvature:} Given a unit normal $\nu$ along a hypersurface $\Sigma$, the second fundamental form $A_g$ is the tangential part of $\nabla \nu$ and the mean curvature $H_g=\Div_\Sigma \nu$. For an ALH manifold, we fix $\nu$ to  point to infinity. 
 \vspace{6pt}

We now state the scalar curvature deformation results in the next two theorems, subject to either boundary condition.  We refer to the precise definitions of the weighted linear spaces in Section~\ref{sec:basic}.  In a loose sense,  $\C^{\ell, \alpha}_{-q}(M)$ consists of $\C^{\ell,\alpha}_{\mathrm{loc}}(M)$ functions or tensors that decay to zero at the rate $q$ toward infinity, and  $\mathcal{B}^{\ell,\alpha}(M)$ consists of $\C^{\ell, \alpha}_{-q}(M)$ functions or tensors that are essentially comparable to $\exp(-\frac{2}{d(x)})$ toward $\Sigma$, where $d(x)$ is the distance function to $\Sigma$.  (In particular, a function (or tensor) in $\mathcal B^{\ell, \alpha}$ decays exponentially to zero  toward $\Sigma$.) We note the number $q\in (\frac{n}{2}, n)$ throughout the paper.

Let  $L_g$ be the linearized scalar curvature at $g$ given by $L_g h:= -\Delta_g (\tr_g h) + \Div_g \Div_g h - h\cdot \Ric_g$ where $h$ is a symmetric $(0,2)$-tensor. 

\begin{Theorem}\label{thm:surjectivity0}
Let $(M, g)$ be an ALH manifold (at rate $q$). Given any scalar function $f\in \mathcal B^{0,\alpha}(M)$, there exists a symmetric $(0,2)$-tensor $h\in \mathcal B^{2,\alpha}(M)$ solving 
\[
L_g h=f.
\] 
Furthermore, there exists $\epsilon>0$ and an open subset $\mathcal{U}\subset g+ \mathcal{B}^{2,\alpha}(M)$ such that given any scalar function $f$ with $\| f \|_{\mathcal{B}^{0,\alpha}(M)} < \epsilon$, there exists $\gamma \in \mathcal{U}$ such that $R_{\gamma} = R_g + f$ in $M$.
\end{Theorem}

Our approach to Theorem~\ref{thm:surjectivity0} uses the variational method for compact manifolds with boundary of J. Corvino~\cite{Corvino:2000td}. But we note that the \emph{no-kernel} assumption in \cite{Corvino:2000td} is not needed in Theorem~\ref{thm:surjectivity0} because we can derive the coercivity estimate on an ALH end without the no-kernel assumption (see Proposition~\ref{prop:exterior estimate}).


The following theorem concerns scalar curvature deformation with prescribed Bartnik boundary data. An analogous statement, see Theorem~\ref{thm:AF} in Appendix~\ref{sec:af},  for asymptotically flat manifolds is established by M.~Anderson and J.~Jauregui~\cite{Anderson:2019tm} for $n=3$. We give a different proof for general dimensions $n\ge 3$ in Appendix~\ref{sec:af}. We use Theorem~\ref{thm:AF}, together with Theorem~\ref{thm:surjectivity0} and a ``gluing'' lemma that constructs an asymptotically flat manifold from an ALH manifold, to prove the following theorem.  
\begin{Theorem}\label{thm:boundary0}
Let $(M, g)$ be ALH at rate $q$. Then the map from  a symmetric $(0,2)$-tensor $h$ 
\[
h\in  \C^{2,\alpha}_{-q}(M)\longmapsto \big( L_g h, h^\intercal, DH|_g(h)\big) \in \C^{0,\alpha}_{-q}(M) \times \C^{2,\alpha}(\Sigma)\times \C^{1,\alpha}(\Sigma)
\]
 is surjective, where $DH|_g$ denotes the linearized mean curvature at $g$.  

By Local Surjectivity Theorem, there exists $\epsilon>0$ and an open subset $\mathcal{U}\subset g+  \C^{2,\alpha}_{-q} (M)$  such that given any $(f, \tau, \phi)$ with $\| (f, \tau, \phi) \|_{\C^{0,\alpha}_{-q}(M) \times \C^{2,\alpha}(\Sigma)\times \C^{1,\alpha}(\Sigma)} < \epsilon$, there exists  $\gamma\in \mathcal{U}$  such that 
\begin{align*}
	R_{\gamma} &= R_g + f\quad \mbox{ in } M\\
	(\gamma^\intercal, H_\gamma)&=(g^\intercal, H_g) + (\tau, \phi)\quad \mbox{ on } \Sigma.
\end{align*}
\end{Theorem}

We give several applications related to the \emph{Wang-Chru\'sciel-Herzlich} mass integrals from those scalar curvature deformation results. (Note that those applications use only the results concerning the linearized operators in Theorem~\ref{thm:surjectivity0} and Theorem~\ref{thm:boundary0}.) We recall that a Riemannian manifold $(M, g)$ is \emph{static} if it admits a function $V$ solving $-\Delta V g + \nabla^2 V - V\,\Ric_g =0$. Such $V$ is called a \emph{static potential}. When $(M, g)$ is ALH with respect to a reference manifold $(\M, \mb)$ that is static with a static potential $V_0$, one can define the deficit of $(M, g)$ from $(\M, \mb)$ by the Wang-Chru\'sciel-Herzlich mass integral $m(g, V_0)$. See Definition~\ref{def:massfunctional}. Our goal is to characterize an ALH manifold that minimizes $m(g, V_0)$ among suitable classes of competitive ALH metrics. 

To set the stage, we let $(M, g)$ be an ALH manifold with boundary $\Sigma$, having mean curvature $H_g \le H_0$ for some function $H_0$ (can be constant). Let $\mathcal{U}$ be a small open neighborhood of $g$ in $g+\C^{2,\alpha}_{-q}(M)$ that contains only Riemannian metrics and define
\begin{equation}\label{eq:neighborhood}
	\begin{split}
    \mathcal{U}_{\star} &= \big\{ \gamma\in \mathcal{U}: (\gamma^\intercal, H_\gamma) = (g^\intercal, H_g)   \mbox{ on } \Sigma \big\} \\
    \mathcal{U}_{\star\star}^{H_0} &=\{ \gamma\in \mathcal{U}: H_\gamma \le H_0 \mbox{ on } \Sigma\}.
  \end{split}
\end{equation}
Note that $\mathcal{U}_{\star}$ is a strictly smaller subset of  $\mathcal{U}_{\star\star}^{H_0} $.

\begin{Theorem}\label{thm:static1}
Let $(\M, \mb)$ be a static, reference manifold with a static potential $V_0$. Let $(M,g)$ be an ALH manifold  with respect to $(\M, \mb)$, having nonempty boundary $\Sigma$. Suppose $(M, g)$ is a mass minimizer in the following sense:
\begin{itemize}  
\item[($\star$)] There is $\mathcal{U}_{\star}$ such that for any $\gamma\in \mathcal{U}_{\star}$ with $R_\gamma = R_g$ in $M$, we have $m(\gamma, V_0)\ge m(g, V_0)$.  
\end{itemize}
 	Then $(M, g)$ is static with a static potential $V$ satisfying $V -V_0=O(r^{-d})$ for some number $d>0$.
\end{Theorem}
 
The proof of Theorem~\ref{thm:static1} relies on a variational argument of R.~Bartnik for the Regge-Teitelboim functional among the scalar curvature constraint  in \cite{Bartnik:2005tn}. That approach has led to other applications, such as establishing the equality case in the spacetime positive mass theorem by the first author and D.~Lee~\cite{Huang:2020um}. See also the recent work of D. Lee, M. Lesourd, and R. Unger~\cite{Lee-Lesourd-Unger:2021} for spacetime positive mass theorem with nonempty boundary.

Separately, Theorem~\ref{thm:static1} has an immediate corollary that a Bartnik mass minimizer must be static. See Section~\ref{sec:Bartnik}.

A version of Theorem~\ref{thm:static1} would also hold for \emph{boundaryless} ALH manifolds by replacing $(\star)$ with the following assumption (by a similar argument as in \cite[Theorem 4.3]{Huang:2020tm}):
\begin{itemize}  
\item[($\star\star\star$)] There is $\mathcal{U}$ such that for any $\gamma\in \mathcal{U}$ with $R_\gamma = R_g$ in $M$, we have $m(\gamma, V_0)\ge m(g, V_0)$. 
\end{itemize}
 A main advantage in the boundaryless case  is that the static potential $V$ must be positive by maximum principle, provided that $V$ is positive near infinity. On the other hand, with nonempty boundary, there are plenty of examples of ALH static manifolds whose static potentials have zeros and are positive near infinity. From that regard, the following theorem is particularly interesting as the mass minimizing condition  $(\star\star)_{H_0}$ enables us to obtain a positive static potential which satisfies additional boundary conditions. 
\begin{Theorem}\label{thm:positive}
Let $(\M, \mb)$ be a static, reference metric with a static potential $V_0>0$ near infinity.	Let $(M,g)$ be an ALH manifold  with respect to $(\M, \mb)$, having nonempty boundary $\Sigma$, and let $H_0$ be a function (can be constant) such that $H_g\le H_0$.  Suppose that
\begin{itemize}  
\item[$(\star\star)_{H_0}$] For any $\gamma\in \mathcal{U}_{\star\star}^{H_0}$ with $R_\gamma = R_g$ in $M$,  we have $m(\gamma, V_0)\ge m(g, V_0)$.
\end{itemize}
Then the following holds:
\begin{enumerate}
\item $(M, g)$ is static with a static potential $V$ satisfying $V -V_0 =O(r^{-d})$ for some number $d>0$.\label{item:static}
\item  $VA_g= \nu(V)g^\intercal$ on $\Sigma$ where  $\nu$ is the unit normal on $\Sigma$ pointing to infinity. \label{item:umbilic}
\item $V>0$ everywhere in $M$.\label{item:positive}
\item  $\Sigma$ has mean curvature $H_g =H_0$.   \label{item:constant}
\end{enumerate}
\end{Theorem}

It is clear that the assumption $(\star\star)_{H_0}$, together with the assumption $H_g\le H_0$, would imply $(\star)$, and thus Item~\eqref{item:static} in Theorem~\ref{thm:positive} directly comes from Theorem~\ref{thm:static1}. So the new content is Items~\eqref{item:umbilic}-\eqref{item:constant}. Our proof is inspired by Anderson-Jauregui~\cite{Anderson:2019tm} (with different analytical details as discussed in Remark~\ref{remark:AJ}) where they analyze the boundary terms from the first variation of the Regge-Teitelboim functional for asymptotically flat manifolds.

Theorem~\ref{thm:positive} also reveals an interesting phenomenon relating to the Penrose inequality for ALH manifolds (see the work  of D.~Lee and A.~Neves~\cite{Lee:2015ue} and references therein). In Lemma~\ref{lem:nonnegative static potential}, we show that a static manifold with the static potential $V>0$ near infinity and having a locally outermost, locally area-minimizing, minimal hypersurface boundary $\Sigma$ must have $V=0$ on $\Sigma$. Note that the generalized Kottler metrics (see Example~\ref{ex:Kottler}) satisfy those properties.  However, by Theorem~\ref{thm:positive}, such a static manifold cannot be a mass minimizer in the sense of $(\star\star)_{0}$ as stated in the following corollary. This corollary can also be compared with Corollary 1.4 in \cite{Lee:2015ue} in which they show that a $3$-dimensional ALH manifold whose boundary is an outermost minimal surface must have the mass strictly greater than the critical mass.

\begin{Corollary}\label{corollary:minimal}
Let $(M, g)$ be a static, ALH manifold with $V>0$ near infinity and having nonempty boundary $\Sigma$. Suppose $\Sigma$ is a locally outermost, locally area-minimizing, minimal hypersurface. Then $(M, g)$ cannot be a mass minimizer, in the sense that given any $\epsilon>0$, there exists an ALH metric $\gamma$ on $M$ such that $\|\gamma-g\|_{\C^{2,\alpha}_{-q}(M)}<\epsilon$,  $R_\gamma=-n(n-1)$, $H_\gamma \le 0$, and $m(\gamma, V)< m(g, V)$. 
\end{Corollary}

The consequences obtained in Theorem~\ref{thm:positive} enable us to establish static uniqueness. There are several existing static uniqueness for ALH manifolds  assuming various boundary conditions by, for example, X. Wang, G.~Galloway and E.~Woolgar, P.~Chru\'sciel, G.~Galloway, and Y.~Potaux~\cite{Wang:2005aa,Galloway:2015tj,Chrusciel:2020wg}.  It appears that the boundary conditions which naturally arise in our mass minimizing problem are different. See Corollary~\ref{corollary:static}. Together with the static uniqueness, Theorem~\ref{thm:positive} gives the following consequence.

\begin{Theorem}\label{thm:uniqueness}
Let $(\M, \mb)$ be a static, reference metric with a static potential $V_0>0$ near infinity.	Let $(M,g)$ be an ALH manifold  with respect to $(\M, \mb)$ possibly with nonempty boundary. Suppose $(M, g)$ has zero mass $m(g, V)=0$ and is a mass minimizer in the following sense:
\begin{itemize}
\item If $M$ has no boundary, we assume ($\star\star\star$) to hold. 
\item If $M$ has boundary $\Sigma$, we assume $H_g\le n-1$ and $(\star\star)_{n-1}$ to hold, i.e. letting $H_0 = n-1$ in the assumption $(\star\star)_{H_0}$. 
\end{itemize}
Then we must have $k=1, 0$, and $(M, g)$ is characterized as follows:
\begin{enumerate}
\item $k=1$: $(M, g)$ is the standard hyperbolic space without  boundary.  
\item  $k=0$: $(M, g)$ is isometric to a Birmingham-Kottler manifold $ \left([1,\infty)\times N, r^{-2} dr^2 + r^2 h\right)$ whose conformal infinity $(N, h)$ is Ricci flat. (In particular, $g$ itself is Poincar\'e-Einstein and the boundary has mean curvature $H_g=n-1$.)
\end{enumerate}
\end{Theorem}

We can put Theorem~\ref{thm:uniqueness} in the more concrete context by using the established positivity in the positive mass theorems.  In the case of asymptotically hyperbolic manifolds,  the reference manifold is the standard hyperbolic space, whose space of static potentials is spanned by  $\{ \sqrt{r^2+1} , x_1, \dots, x_n\}$. They give  the mass integrals 
\[
p_0 = m(g, \sqrt{r^2+1} )\quad \mbox{ and }\quad p_i= m(g, x_i),\quad i=1,\dots, n. 
\]
There has been much progress toward the positive mass theorem for asymptotically hyperbolic manifolds, by X. Wang~\cite{Wang:2001ww}, P. Chru\'sciel~and M.~Herzlich~\cite{Chrusciel:2003aa}, L.~Andersson, M.~Cai, and G.~Galloway~\cite{Andersson:2008vy}, Chru\'sciel and E.~Delay~\cite{Chrusciel.2019}, D. Martin with the authors~\cite{Huang:2020tm},  A.~Sakovich~\cite{Sakovich:2021vu}, and Chru\'sciel-Galloway~\cite{Chrusciel:2021}. The combined efforts give the following statement. 
\begin{Theorem}[\cite{Wang:2001ww,Chrusciel:2003aa,Andersson:2008vy,Chrusciel.2019,Chrusciel:2021,Huang:2020tm,Sakovich:2021vu}]\label{thm:pmt1}
	Let  $3\le n\le 7$ and $(M,g)$ be an $n$-dimensional asymptotically hyperbolic manifold (possibly with nonempty boundary)\footnote{For the case of boundaryless $M$, the theorem also holds  for spin manifolds in dimensions $n\ge 3$. The spin assumption can be removed if the spacetime positive mass theorem for asymptotically flat initial data sets holds in those dimensions (see \cite{Lohkamp:2017tg}).} with scalar curvature $R_g\ge -n(n-1)$. In the case that $M$ has a nonempty boundary $\Sigma$, we assume $\Sigma$ has mean curvature $H_g\le n-1$. Then $p_0 \ge \sqrt{p_1^2+\dots + p_n^2}$. 
	
If $(M, g)$ has no boundary and $p_0 = \sqrt{p_1^2+\dots + p_n^2}$, then $(M, g)$ is isometric to hyperbolic space. 
\end{Theorem}

Theorem~\ref{thm:uniqueness} and Theorem~\ref{thm:pmt1} give the following consequence. 

\begin{Theorem}\label{thm:rigidity1}
Let $(M, g)$ satisfy the assumptions in Theorem~\ref{thm:pmt1}. If $p_0 = \sqrt{p_1^2+\dots + p_n^2}$, then $(M, g)$ has no boundary and thus is isometric to hyperbolic space. 
\end{Theorem}

Positivity of the mass for general ALH manifolds seems to be less understood (perhaps more intriguing at the same time) in light of examples of the generalized Kottler metrics with negative mass when $k=-1$.  Assuming the conformal infinity is either the flat torus or a nontrivial quotient of the unit sphere, the following positivity of mass is proven by Chru\'sciel, Galloway, L.~Nguyen, and T.~Paetz \cite[Theorem 1.1]{Chrusciel:2018vp} (see also the remark from \cite[Theorem VI.4]{Chrusciel:2020wg} on the mean curvature assumption).
\begin{Theorem}[Chru\'sciel-Galloway-Nguyen-Paetz~\cite{Chrusciel:2018vp}]\label{thm:pmt0}
	Let  $3\le n\le 7$. Let  $(M,g)$ be an $n$-dimensional ALH manifold with conformal infinity $(N, h)$ and $M$ be diffeomorphic to $[1, \infty)\times N$.  Suppose that: 
	\begin{enumerate}
		\item The  boundary $\Sigma = \{ 1\} \times N$ has mean curvature $H_g\le n-1$.
		\item The scalar curvature satisfies $R_g\ge -n(n-1)$.
		\item The conformal infinity $(N,h)$ is either a flat torus or a nontrivial quotient of the standard sphere. Define the mass $m(g) =m(g, \sqrt{r^2})$ (may be up to a positive normalizing constant) in the former case and $m(g)=m(g, \sqrt{r^2+1})$ in the later case. 
		\item When $n=3$, assume the mass aspect function has a sign. 
	\end{enumerate}
Then $m(g)\ge 0$.
\end{Theorem}

Theorem~\ref{thm:uniqueness} and Theorem~\ref{thm:pmt0} give the following consequence. Note that the following theorem implicitly implies that if $(M, g)$ has boundary with mean curvature $\le n-1$ and if $(N, h)$ is a nontrivial quotient of the standard sphere, then necessarily $m(g)>0$.  

\begin{Theorem}\label{thm:rigidity0}
Let $(M, g)$ satisfy the assumptions in Theorem~\ref{thm:pmt0}. If $m(g)=0$, then $(M, g)$ must be a locally hyperbolic manifold whose conformal infinity $(N, h)$ is the flat torus. Namely, $(M, g)$ is isometric to $\left( [1,\infty)\times N, r^{-2} dr^2 + r^2 h\right)$. 
\end{Theorem}

\noindent{\bf Acknowledgement.} We thank Zhongshan An for reading Appendix~\ref{sec:af} and for valuable comments. We are very grateful to the referee for thorough reading and for helpful comments to significantly improve the paper.

\section{Preliminaries}\label{se:pr}






In this section, we discuss background materials, definitions, and basic facts. Most parts of this section are already known to the experts. The main purpose is to set the stage for later sections, as we will frequently refer back to this section. We may include some proofs to extend classical results to our current setting of ALH manifolds. 

\subsection{Reference manifolds and Birmingham-Kottler metrics}

Let  $(N, h)$ be an  $(n-1)$-dimensional connected, closed (compact without boundary) Riemannian manifold. From it, we can define an $n$-dimensional manifold $\M:=(r_k, \infty)\times N$ with the metric 
\begin{align}\label{eq:reference1}
	\mb = \frac{1}{r^2+k} dr^2 + r^2 h,
\end{align}
where $k$ is a constant normalized so that $k\in \{ -1, 0, 1\}$. We also set  $r_k =0$ for $k=0, 1$, and $r_k=1$ for $k=-1$. By changing the coordinate  $s= \frac{2}{r+\sqrt{r^2+k}}$, the metric $\mb$  can be alternatively re-expressed as   
\[
	\mb= s^{-2} \Big(ds^2 + \big(1-\tfrac{k}{4} s^2\big)^2 h\Big).
\] 
 Because the metric $\overline{\mb} := ds^2 + \big(1-\tfrac{k}{4} s^2\big)^2 h$ can be trivially extended on $\{s=0\}\times N$ with the induced metric $h$ (which corresponds to the ``infinity'' $r=\infty$), the metric~$\mb$ is said to be \emph{conformally compact}, and the ``extended'' boundary-at-infinity  $(N, h)$  is called the \emph{conformal infinity}. One computes $|ds|_{\overline{\mb}}^2=1$, and thus a computation says that the metric~$\mb$ has asymptotically constant sectional curvature $-1$ as $s\to 0$ (see \cite[p. 192]{Graham:1991aa} and \cite{Mazzeo:1988ta}). For that reason,  the metric $\mb$ is said to be \emph{asymptotically locally hyperbolic} (or \emph{ALH} for short).

\vspace{6pt}
\noindent{\bf Convention:} We refer to $(\M, \mb)$ defined above as a \emph{reference manifold} (of type $k$ and with conformal infinity $(N, h)$). We denote by $B_r = (r_k, r] \times N$ and $S_r = \{ r \} \times N$ the closed subsets in $(\M, \mb)$ respectively as the coordinate ``ball'' and ``sphere'' (although they may not be topologically a ball nor a sphere).

We discuss some basic properties of a reference manifold and formulate them as ``examples'' below, so we can refer them later. 
\begin{example}[Reference manifolds in geodesic polar coordinates]\label{ex:1}
While the reference manifold written in the $r$-coordinate in~\eqref{eq:reference1} may seem to have singularity at $r=r_k$,  the manifold may be made complete (possibly after extension). To see this, we  change the $r$-coordinate in ~\eqref{eq:reference1} to the ``geodesic'' coordinate $t$:
\begin{itemize}
\item $k=1$. We let $t=\sinh^{-1}(r)$ and rewrite $\mb$ as 
\[
	\mb = dt^2+ (\sinh t)^2 h,
\] 
defined on $t\in (0, \infty)$. When the conformal infinity $(N, h)$ is the standard unit sphere, $(\M, \mb)$ is just the standard hyperbolic space and can be extended at $r=0$ (or $t=0$) as a complete manifold without boundary.
\item $k=0$. We let $t= \ln r$. After changing coordinate from $r$ to $t$,  
\[
	\mb = dt^2 + e^{2t} h 
\]
which is defined on $t\in (-\infty, \infty)$. The reference manifold has an ``expanding'' end as $t\to \infty$ (i.e. $r\to \infty)$ and the other end is ``shrinking'' as $t\to -\infty$ (i.e. $r\to 0$), which is usually called the hyperbolic cusp.  See Figure~\ref{figure1}.
\item $k=-1$. We let $t=\cosh^{-1}(r)$ and write the metric in the $t$-coordinate as $\mb=dt^{2}+(\cosh t)^{2}h$. The metric can be obviously extended from its original range  $t\in (0,\infty)$, i.e. $r\in (1,\infty)$, to $t\in (-\infty, \infty)$. The extended manifold 
\begin{align}\label{eq:extension}
	\M_{\mathrm{ext}} = (-\infty, \infty) \times N \quad \mbox{ and } \quad  \mb_{\mathrm{ext}} =dt^{2}+(\cosh t)^{2}h
\end{align}
is complete, has reflectional symmetry with respect to $t\mapsto -t$, and has two isometric ALH ends. The ``neck'' $\{ t=0\}$, i.e. $\{ r=1\}$, is a minimal hypersurface. See Figure~\ref{figure2}.
\end{itemize}
\qed
\end{example}

\begin{example}[Constant mean curvature foliations]\label{ex:CMC}
The reference manifold is foliated by $r$-level hypersurfaces of constant mean curvature. In fact, each $S_r$ is umbilic, and the second fundamental form is given by $A_g = \frac{\sqrt{r^2+k}}{r} \mb^\intercal$. In particular, the mean curvature $H_{S_r}$ of $S_r$ satisfies
\begin{itemize}
\item $k=1$: \quad $H_{S_r} = (n-1)\frac{\sqrt{r^2+1}}{r}$, strictly decreasing to $(n-1)$ as $r\to \infty$.\vspace{3pt}
\item $k=0$:\quad $H_{S_r}=n-1$ for any $r$. \vspace{3pt}
\item $k=-1$:\quad  $H_{S_r}=\frac{\sqrt{r^2-1}}{r}(n-1)$, strictly increasing to $(n-1)$ as $r\to \infty$.
\end{itemize}
\qed
\begin{figure}[ht]  
  \begin{minipage}{.45\textwidth}
    \def\svgwidth{2.6in}
\begingroup%
  \makeatletter%
  \providecommand\color[2][]{%
    \errmessage{(Inkscape) Color is used for the text in Inkscape, but the package 'color.sty' is not loaded}%
    \renewcommand\color[2][]{}%
  }%
  \providecommand\transparent[1]{%
    \errmessage{(Inkscape) Transparency is used (non-zero) for the text in Inkscape, but the package 'transparent.sty' is not loaded}%
    \renewcommand\transparent[1]{}%
  }%
  \providecommand\rotatebox[2]{#2}%
  \newcommand*\fsize{\dimexpr\f@size pt\relax}%
  \newcommand*\lineheight[1]{\fontsize{\fsize}{#1\fsize}\selectfont}%
  \ifx\svgwidth\undefined%
    \setlength{\unitlength}{584.40401333bp}%
    \ifx\svgscale\undefined%
      \relax%
    \else%
      \setlength{\unitlength}{\unitlength * \real{\svgscale}}%
    \fi%
  \else%
    \setlength{\unitlength}{\svgwidth}%
  \fi%
  \global\let\svgwidth\undefined%
  \global\let\svgscale\undefined%
  \makeatother%
  \begin{picture}(1,0.73860793)%
    \lineheight{1}%
    \setlength\tabcolsep{0pt}%
    \put(0,0){\includegraphics[width=\unitlength,page=1]{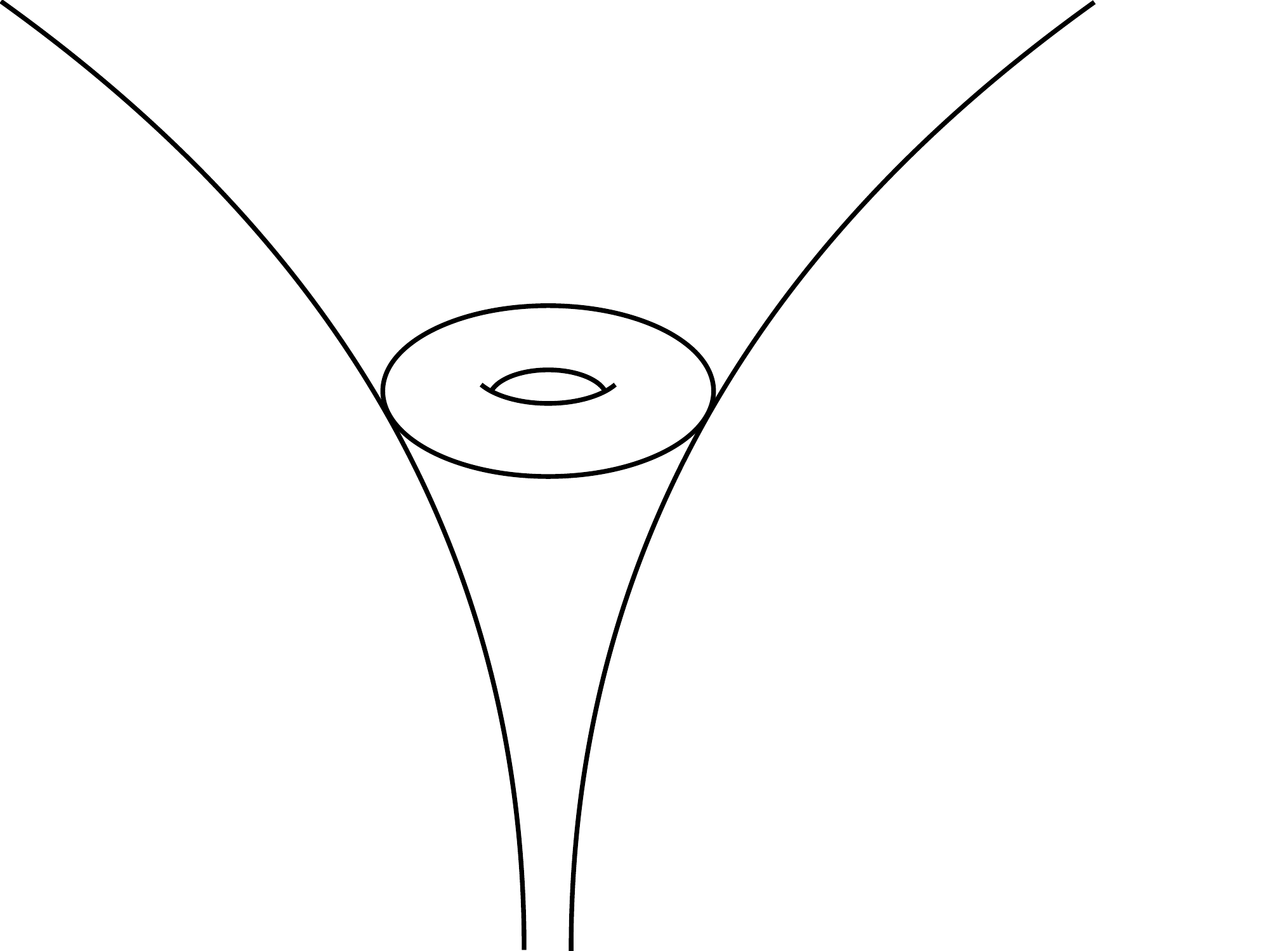}}%
    \put(0.60877998,0.40960882){\makebox(0,0)[lt]{\lineheight{1.25}\smash{\begin{tabular}[t]{l}$t=0$\end{tabular}}}}%
    \put(0.86271144,0.67550655){\makebox(0,0)[lt]{\lineheight{1.25}\smash{\begin{tabular}[t]{l}$t\to +\infty$\end{tabular}}}}%
    \put(0.49623146,0.0162757){\makebox(0,0)[lt]{\lineheight{1.25}\smash{\begin{tabular}[t]{l}$t\to -\infty$\end{tabular}}}}%
  \end{picture}%
\endgroup%

    \caption{$k=0$}
            \label{figure1}
  \end{minipage}
  \begin{minipage}{.45\textwidth}
    \def\svgwidth{2.6in}
\begingroup%
  \makeatletter%
  \providecommand\color[2][]{%
    \errmessage{(Inkscape) Color is used for the text in Inkscape, but the package 'color.sty' is not loaded}%
    \renewcommand\color[2][]{}%
  }%
  \providecommand\transparent[1]{%
    \errmessage{(Inkscape) Transparency is used (non-zero) for the text in Inkscape, but the package 'transparent.sty' is not loaded}%
    \renewcommand\transparent[1]{}%
  }%
  \providecommand\rotatebox[2]{#2}%
  \newcommand*\fsize{\dimexpr\f@size pt\relax}%
  \newcommand*\lineheight[1]{\fontsize{\fsize}{#1\fsize}\selectfont}%
  \ifx\svgwidth\undefined%
    \setlength{\unitlength}{593.94317387bp}%
    \ifx\svgscale\undefined%
      \relax%
    \else%
      \setlength{\unitlength}{\unitlength * \real{\svgscale}}%
    \fi%
  \else%
    \setlength{\unitlength}{\svgwidth}%
  \fi%
  \global\let\svgwidth\undefined%
  \global\let\svgscale\undefined%
  \makeatother%
  \begin{picture}(1,0.64785951)%
    \lineheight{1}%
    \setlength\tabcolsep{0pt}%
    \put(0,0){\includegraphics[width=\unitlength,page=1]{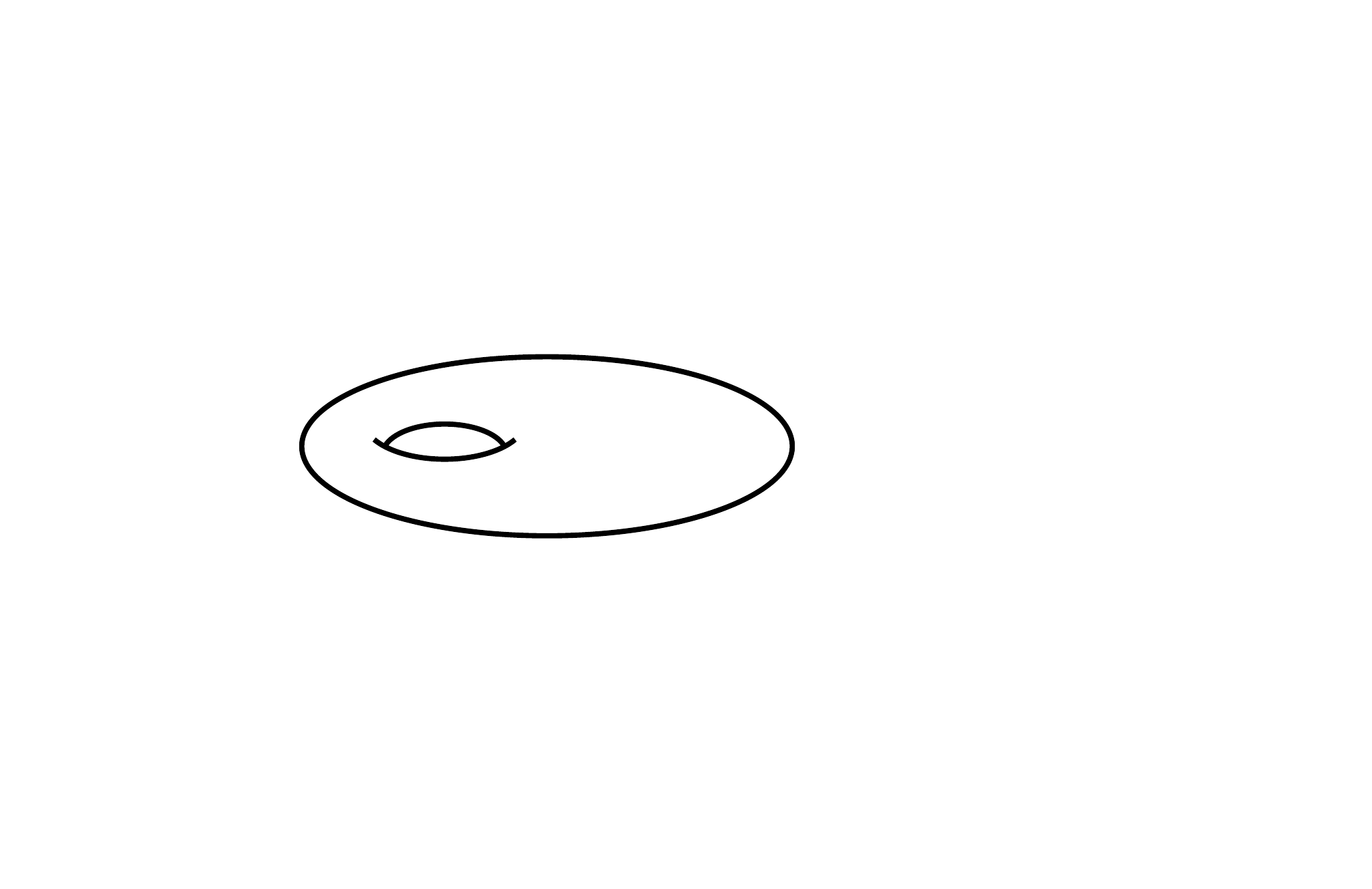}}%
    \put(0.63814208,0.31473756){\makebox(0,0)[lt]{\lineheight{1.25}\smash{\begin{tabular}[t]{l}$t=0 \text{ (minimal)}$\end{tabular}}}}%
    \put(0.80971906,0.58552097){\makebox(0,0)[lt]{\lineheight{1.25}\smash{\begin{tabular}[t]{l}$t\to +\infty$\end{tabular}}}}%
    \put(0.8046661,0.04631754){\makebox(0,0)[lt]{\lineheight{1.25}\smash{\begin{tabular}[t]{l}$t\to -\infty$\end{tabular}}}}%
    \put(0,0){\includegraphics[width=\unitlength,page=2]{fig2_type-1.pdf}}%
  \end{picture}%
\endgroup%

    \caption{$k=-1$}
    \label{figure2}
  \end{minipage}
\end{figure}

\end{example}

Consider the scalar curvature  as a map from a Riemannian metric~$g$ on a smooth manifold $M$ to  its scalar curvature $R_g$. The \emph{linearized scalar curvature map} $L_g:=DR|_g$ at a metric $g$ is given by the formula 
\[
L_g h=-\Delta_g (\tr _g h) + \mathrm{div}_g \mathrm{div}_g h - h \cdot \Ric_g
\] 
where $h$ is a symmetric $(0,2)$-tensor. (We note that $\Delta_g$ is the trace of the Hessian.)  Denote by $L_g^*$ the formal $L^2$-adjoint operator of $L_g$. For a scalar function $V$, 
\begin{align}\label{eq:static}
 L_g^* V=(-\Delta_g V)\, g + \nabla_g^2 V -V \Ric_g.
\end{align}
A Riemannian manifold $(M, g)$  is said to be \emph{static} if it admits a nontrivial function $V$ solving $L_g^* V=0$, and such $V$ is called a \emph{static potential}. Note that a connected, static manifold necessarily has constant scalar curvature, and thus an $n$-dimensional static, ALH manifold must have constant scalar curvature $R_g = -n(n-1)$. Then taking the trace of $L_g^*V=0$ and replacing the Laplace term, the static equation $L_g^*V=0$ is equivalent to 
\begin{align}\label{eq:static-system}
\begin{split}
	\nabla^2 V - V(\Ric_g + ng) &=0\\
	\Delta_g V-nV&=0. 
\end{split}
\end{align}
While Section~\ref{sec:scalar1} and Section~\ref{sec:scalar2}  hold for any reference manifolds regardless of the geometry nor topology of $(N, h)$, we will work on a reference manifold that is \emph{static} in Section~\ref{sec:static} and Section~\ref{sec:uniqueness}.  So it may be helpful to have concrete examples of static reference manifolds in mind. 
\begin{example}[Birmingham-Kottler metrics]
We say a reference manifold $(\M, \mb)$ is  \emph{Birmingham-Kottler} if $ \mb$ is Poincar\'e-Einstein, i.e. $\Ric_{\mb} = -(n-1) \mb$. It is equivalent to requiring that its conformal infinity $(N, h)$ has constant Ricci curvature, $\Ric_h = k (n-2)  h$. A Birmingham-Kottler manifold $(\M, \mb)$ is static with a static potential 
\[
	V = \sqrt{r^2+ k}.
\] 
It is discussed in \cite[Appendix B]{Chrusciel:2001vf} that for the case $k=0, -1$ any static potential is a multiple of $V$, and for the case $k=1$ there may be additional static potentials from the isometries of $(N, h)$. For example, when $(N, h)$ is a round sphere $S^{n-1}$, the additional static potentials are $\{x_1, \dots, x_{n}\}$, the restriction  of the Cartesian coordinates in $\mathbb{R}^n$ on $S^{n-1}$. Since $\mb$ is Poincar\'e-Einstein, for any static potential $V$, we shall see that the mass integral, as defined in Definition~\ref{def:massfunctional} below, is zero, and also that $V$ satisfies the Obata equation:
\begin{align}
	m(\mb, V) &= 0 \\
	\nabla^2_{\mb} V &= V \mb.
\end{align}
\qed
\end{example}

In view of Corollary~\ref{corollary:minimal}, we give examples of static manifolds whose boundary is an outermost, area-minimizing minimal hypersurface. 
\begin{example}[Generalized Kottler metrics]\label{ex:Kottler}
Let $(N, h)$ be a closed $(n-1)$-dimensional Riemannian manifold with constant Ricci curvature $\Ric_h = k (n-2) h$. Define the metric 
\[
	g_m = \frac{1}{r^2+k-\frac{2m}{r^{n-2}} }dr^2 + r^2 h
\]
on $(r_m, \infty)\times N$ where $r_m 
$ is the largest zero of $r^2+k-\frac{2m}{r^{n-2}}=0$. If $r_m>0$, the manifold has an outermost minimal hypersurface $\Sigma=\{ r_m\} \times N$. The manifold  is static with a static potential
\[
	V=\sqrt{r^2+k-\frac{2m}{r^{n-2}}}.
\] 
From computing as in Definition~\ref{def:massfunctional}  below, the mass $m(g_m,V)=2(n-1)m$ where the volume of $(N,h)$ is normalized as $1$. We note that when $k = -1$, it is possible to have examples of generalized Kottler metrics with $r_m > 0$ and negative mass $m < 0$.
\qed
\end{example}
\subsection{ALH manifolds and weighted function spaces}\label{sec:basic}

Before we specify the class of ALH metrics considered in this paper, we first define the weighted H\"older norms (with respect to a fixed reference manifold $(\M, \mb)$). 
\begin{definition}
 For $ \alpha \in [0, 1)$, $\ell=0, 1, 2, \dots$, and $q\in \mathbb{R}$, the weighted H\"older space $\C^{\ell,\alpha}_{-q} (\M\setminus B_2)$ is  the collection of $\C^{\ell,\alpha}_{\mathrm{loc}}(\M\setminus B_2)$-functions $f$ that satisfy
\begin{align*}
		\| f \|_{\C^{\ell,\alpha}_{-q}(\M\setminus B_2)} := \sum_{|I|=0, 1,\dots, k}\sup_{x\in \M\setminus B_2} r^q |\nabla^I f(x)| +\sup_{x\in \M\setminus B_2} r^{q+\alpha} [\nabla^\ell f]_{\alpha; B_1(x)} <\infty,
\end{align*}
where the covariant derivatives and norms are taken with respect to the reference metric $\mb$, $B_1(x)\subset \M$ is a geodesic unit ball centered at $x$, and  
\[
	[ \nabla^\ell f]_{\alpha; B_1(x)} := \sup_{1\le i_1, \dots, i_\ell\le n} \sup_{y\neq z\in B_1(x)} \frac{|e_{i_1} \cdots e_{i_\ell}(f) (y) - e_{i_1}\dots e_{i_\ell}(f) (z)|}{ d(y, z)^\alpha}
\] 
with respect to a fixed local orthonormal frame $\{ e_1,\dots, e_n\}$ on $(\M, \mb)$, where $e_1 = \sqrt{r^2+k} \frac{\partial}{\partial r}$, $e_\alpha = r^{-1} \hat{e}_\alpha$ for  $\alpha=2, \dots, n$ and $\{ \hat{e}_2, \dots, \hat{e}_n\}$ is a local orthonormal frame of $(N, h)$.

Let $M$ be an $n$-dimensional smooth manifold such that there is a compact set $K\subset M$ and a diffeomorphism $\Phi: \M \setminus B_2 \to M\setminus K$.	For a function $f$, we define the \emph{weighted H\"older norm} $\|f\|_{\C^{\ell,\alpha}_{-q}(M)}$ to be the sum of the weighted norm $\|f\|_{\C^{\ell,\alpha}_{-q}(M\setminus K)} := \| f\circ \Phi \|_{\C^{\ell,\alpha}_{-q}(\M\setminus B_2)}$ and the usual $\C^{\ell,\alpha}$ norms on the compact set $K$ with respect to  a fixed atlas on $M$. Define the space $\C^{\ell,\alpha}_{-q}(M)$ to be the completion of smooth compactly supported functions with respect to the $\C^{\ell,\alpha}_{-q}(M)$-norm.  We shall also write $O^{\ell,\alpha}(r^{-q})$ for functions in $\C^{\ell, \alpha}_{-q}(M)$ to emphasize the fall-off rate, and simply $O(r^{-q})$ for $\C^{0}_{-q}(M)$. With respect to the fixed atlas, we can extend the definition to tensor fields. For instance, we say a tensor field is $\C^{\ell,\alpha}_{-q}(M)$ if all of its components belong to $\C^{\ell,\alpha}_{-q}(M)$.
\end{definition}

\begin{definition}\label{def:locallyhyperbolic}
Let $(M,g)$ be a  smooth, connected $n$-dimensional manifold and $g$ be $\C^\infty_{\mathrm{loc}}$. We say that $(M,g)$ is \emph{asymptotically locally hyperbolic (ALH) (at rate $q$ and with respect to the reference manifold $(\M, \mb)$)} if there exists a compact set $K\subset M$ and a diffeomorphism $\Phi: \M \setminus B_2 \to M\setminus K$ such that, with respect to the coordinate chart of $\Phi$, 
\begin{align}
		g_{ij} -\mb_{ij} &\in \C^{2,\alpha}_{-q} (M\setminus K)\\
		r(R_g + n(n-1))&\in L^1(M\setminus K).\label{eq:L1}
\end{align}
We will slightly blur the distinction between $\M \setminus B_2$ and  $M\setminus K$. For example, we write $\mb$ in $M\setminus K$ also for the pull-back metric from $(\M\setminus B_2, \mb)$ and denote by $B_r, S_r$ for $r\gg 1$ as the subsets in $M$ such that $M\setminus B_r$ and $S_r$ are the $\Phi$-image sets of $\M \setminus B_r$ and $S_r$ in $\M$, respectively.

If the reference manifold is the standard hyperbolic space, $(M, g)$ is called \emph{asymptotically hyperbolic}.
\end{definition}

Throughout the paper, we will assume $\alpha\in (0, 1)$ and the fall-off rate $q$ to stay in the range:
\[
	\frac{n}{2}<q<n.
\]	
The lower bound ensures that the mass integral in Definition~\ref{def:massfunctional} below is well-defined, while the upper bound is to avoid the ``critical exponent'' that usually comes up in analysis as in Section~\ref{sec:analysis}.

For an ALH manifold $(M, g)$ with nonempty boundary $\Sigma$, we will consider the weighted spaces to describe functions that decay to zero exponentially fast toward $\Sigma$.  Let $d(x)$ be the distance function to~$\Sigma$, defined in a collar neighborhood of $\Sigma$.  We define a weight function $\rho(x)$ on $M$ such that $\rho>0$ in $\Int M$ and 
\[
	\rho(x)=\left\{
		\begin{array}{ll}
			\exp \left(-\frac{4}{d(x)}\right) & \text{ in a thin collar neighborhood of $\Sigma$}\vspace{5pt}\\
			{r(x)^{-2q+n-\delta}} & \text{ in } M\setminus K
		\end{array}
	\right..
\]
where $\delta$ is a fixed number such that 
\begin{align}\label{eq:d}
	\delta \in \left(\frac{2 (\sqrt{n^2+n-2} - n)}{n-2}, 1\right).
\end{align}
 Note that the lower bound $\frac{2 (\sqrt{n^2+n-2} - n)}{n-2}$ is a number between $0$ and $1$ for any $n\ge 3$ and   is used for the estimates in the proof of Proposition~\ref{prop:exterior estimate} below.
Define the weighted  $L^2_\rho(M)$ norm by 
\begin{align*}
	\| u \|_{L^2_\rho(M)}^2 &=\int_M |u|^2 \rho \, d\mu_g.
\end{align*}
 The dual space  $L^2_{\rho^{-1}}(M)$ to $L^2_\rho(M)$ is the main interest here. Roughly speaking,  $L^2_{\rho^{-1}}(M)$ contains functions that (in an integral sense) decay  to zero at the rate $\exp \big(-\frac{2}{d(x)}\big)$ toward $\Sigma$ and decay to zero toward infinity at a rate just a little bit slower than $r^{-q}$.   More specifically, the assumption $\delta<1$ implies 
 $\C^{0}_{-q}(M\setminus B_R)\subset L^2_{\rho^{-1}}(M\setminus B_R)$ (noting the volume measure $d\mu_g = \big(1+o(1)\big) \frac{r^{n-1}}{\sqrt{r^2+k}} \, dr d\sigma_h = O(r^{n-2}) \, dr d\sigma_h$). We also define the $H^k_{\rho}(M)$-norm by 
\[	
	\| u \|_{H^k_\rho (M)}^2 = \sum_{|I|=0}^k \big\| |\nabla^I_g u|_g \big\|_{L^2_\rho(M)}^2.
\]
Let $H^k_\rho(M)$ be the completion of $\C^\infty_c(M)$ with respect to the $H^k_\rho(M)$-norm. For any  $B_r \subset M$ with~$r$ sufficiently large, we can decompose the integral over $B_r$ and the exterior $M\setminus B_r$:
\begin{align}\label{eq:norm}
	\| u \|_{H^k_\rho(M)}^2&= \sum_{|I|=0}^k \int_{B_r} |\nabla^I u|^2 \rho \, d\mu_g +  \sum_{|I|=0}^k\int_{M\setminus B_r} |\nabla^I u|^2  r^{-2q + n-\delta} \, d\mu_g\notag\\
	&=:\| u \|_{H^k_\rho(B_r)}^2 + \| u \|_{H^k_\rho(M\setminus B_r)}^2.
\end{align}
 Clearly, we have 
\begin{align}
\label{eq:exterior}
\begin{split}
	\| u \|_{H^k_\rho(M)} &=\left(\| u \|_{H^k_\rho(B_r)}^2 + \| u \|_{H^k_\rho(M\setminus B_r)}^2\right)^{\frac{1}{2}}\\
	&\le \left (\| u \|_{H^k_\rho(B_r)}+ \| u \|_{H^k_\rho(M\setminus B_r)}\right)\le 2\| u \|_{H^k_\rho(M)}.
\end{split}
\end{align}
Those Hilbert spaces will be used to describe weak solutions to an elliptic PDE in the variational argument in Section~\ref{sec:scalar1}.

We will also use the weighted H\"older spaces to describe $\C^{\ell,\alpha}_{\mathrm{loc}}(M)$-functions that can be, roughly speaking,  bounded by the weight function $\rho^{\frac{1}{2}}$. We follow closely the definition $\C^{\ell, \alpha}_{\phi, \varphi}$ from \cite{Corvino:2020ty} (see also \cite{Chrusciel:2003ug}):
\[
	\| u \|_{\C^{\ell, \alpha}_{\phi, \varphi} (M)} =  \sup_{x\in M} \left(\sum_{j=0}^\ell \varphi(x) \phi^j(x) \| \nabla_g^j u \|_{\C^0(B_{\phi(x)}(x))} + \varphi(x) \phi^{\ell+\alpha}(x) [\nabla^\ell_g u]_{ \alpha; B_{\phi(x)}(x)}\right)
\]
where the ``scaling'' $\phi$ is a function on $M$ such that $\phi>0$ in $\Int M$, $\phi(x) =d(x)^2$ near $\Sigma$, and $\phi(x)=1$ outside a larger compact set\footnote{The scaling function $\phi$ only appears in this section and the reader should not confuse it with the same notation that is used differently to prescribe the mean curvature in Theorem~\ref{thm:boundary0} and the rest of the paper.}. We define 
\begin{align*}
	\mathcal{B}^{2,\alpha}(M) &= \C^{2, \alpha}_{\phi, \phi^{2+\frac{n}{2}}  \rho^{-\frac{1}{2} }}(M) \cap L^2_{\rho^{-1}}(M)\cap \C^{2,\alpha}_{-q}(M)\\
	\mathcal{B}^{0,\alpha}(M) &= \C^{0, \alpha}_{\phi, \phi^{4+\frac{n}{2}}  \rho^{-\frac{1}{2} }}(M) \cap L^2_{\rho^{-1}}(M)\cap \C^{0,\alpha}_{-q}(M).
\end{align*} 

The definition of $\C^{\ell, \alpha}_{\phi, \varphi} (M)$ is quite technical, but we will not explicitly use its definition but just the results stated in those norms from \cite{Corvino:2020ty}. Therefore, it suffices to recognize some general properties in order to make sense of the proofs in Section~\ref{sec:scalar1}.  In particular, $\mathcal{B}^{\ell,\alpha}(M)$ is a subset of $\C^{\ell, \alpha}_{\mathrm{loc}}$-functions whose $(\ell,\alpha)$-th derivatives,  for $\ell=0, 1,2$, decay to zero toward $\Sigma$ exponentially (at least at the rate $\exp \left(-\tfrac{2}{d(x)}\right)$) and  decay to zero toward infinity at least at the rate $r^{-q}$. 

\subsection{Basic analysis}\label{sec:analysis}

To study the scalar curvature map on an ALH manifold, we will frequently encounter the operator $\Delta_g  - n$ on scalar functions and will use the following result. 
\begin{lemma}\label{lemma:isomorphism}
Let  $s\in (-1, n)$ and $(M, g)$ be an $n$-dimensional ALH manifold with nonempty boundary $\Sigma$. Denote $\mathsf{T}u= \Delta_g u - nu$.  Then the following holds:
\begin{enumerate}
\item $\mathsf{T}: \Big\{ u\in \C^{2,\alpha}_{-s}(M): u = 0 \mbox{ on } \Sigma\Big\} \to \C^{0,\alpha}_{-s}(M)$ is an isomorphism.\label{item:Dirichlet}
\item $\mathsf{T}: \Big\{ u\in \C^{2,\alpha}_{-s}(M): \nu(u) = 0 \mbox{ on } \Sigma\Big\} \to \C^{0,\alpha}_{-s}(M)$ is an isomorphism, where $\nu$ is a unit normal vector on $\Sigma$. \label{item:Neumann}
\end{enumerate}
Therefore, given any  $f\in \C^{0,\alpha}_{-s}(M)$ and $\eta$, where $\eta\in \C^{2,\alpha}(\Sigma)$ for the Dirichlet boundary problem and  $\eta\in \C^{1,\alpha}(\Sigma)$ for the Neumann boundary problem, there exist unique $u_1, u_2\in  \C^{2,\alpha}_{-s}(M)$ such that 
\begin{align*}
	\mathsf{T} u_1 &= f \mbox{ in $M$  with $u_1=\eta$ on $\Sigma$},\\
	\mathsf{T}u_2 &= f \mbox{ in $M$  with $\nu(u_2)=\eta$ on $\Sigma$}.
\end{align*}
As a direct consequence, if $f$ is compactly supported, then $u_1, u_2\in \C^{2,\alpha}_{-s}(M)$ for any $s<n$. 
\end{lemma}
\begin{remark}\label{rmk:Neumann}
We shall use an equivalent statement to Item~\eqref{item:Neumann} that $ u\in \C^{2,\alpha}_{-s}(M)$ to $(\mathsf{T}u, \nu(u)) \in \C^{0,\alpha}_{-s}(M)\times \C^{1,\alpha}(\Sigma)$ is an isomorphism.  We also note that if $M$ is boundaryless, the same argument shows that $\mathsf{T}:  \C^{2,\alpha}_{-s}(M) \to \C^{0,\alpha}_{-s}(M)$ is an isomorphism. 
\end{remark}
\begin{proof}
The proof follows similar arguments as in \cite[Section 3]{Graham:1991aa} where they study conformally compactifiable manifolds whose conformal infinity is a sphere and, in the case of non-empty boundary, with the Dirichlet boundary condition.  (Note that our $\Delta_g$ is the trace of the Hessian and has the opposite sign from that in \cite{Graham:1991aa}. Also, their manifold is of dimension $n+1$.)  The numbers $\{-1, n\}$ that defines the range of $s$ are known as the characteristic exponents. Each of them corresponds to the rate that a nontrivial solution to $\mathsf{T}u=0$ can exist so the isomorphism theorem would fail if $s= -1$ or $n$.

We now describe how slight modifications of their arguments can prove the lemma in the ALH setting and for the Neumann boundary condition. As in \cite[Proposition 3.9]{Graham:1991aa}, for any fixed $s\in (-1, n)$, one can construct a bounded \emph{positive} radial function $\varphi(x)$ on $M$ such that $\varphi(x) = r^{-s}$ near infinity, $\varphi$ is constant in a collar neighborhood of $\Sigma$,  and  $\mathsf{T}\varphi < -\epsilon \varphi$ for some constant $\epsilon>0$. Then we will proceed as in \cite[Proposition 3.8]{Graham:1991aa} to derive the \emph{basic estimate}: there is a constant $C>0$ such that for  $u\in \C^{2,\alpha}_{-s}(M)$ with $u = 0$ on $\Sigma$, 
\[
\| u \|_{\C^{0}_{-s}(M)}\le C \| \mathsf{T} u \|_{\C^{0}_{-s}(M)}.
\] 

We sketch their proof so we can highlight the necessary modification for the Neumann boundary condition in the next paragraph. It is direct to compute that 
\begin{align} \label{eq:diff}
	\Delta_g \left(\frac{u}{\varphi} \right)= \frac{\mathsf{T}u}{\varphi} - 2\nabla \log \varphi \cdot \nabla \left(\frac{ u}{\varphi}\right) - \frac{u}{\varphi} \left(\frac{\mathsf{T} \varphi }{\varphi} \right).
\end{align}
Because $\varphi>0$ and $\varphi=r^{-s}$ near infinity,  $\frac{u}{\varphi}$ is bounded in $M$.  We may assume without loss of generality that $\sup_M \frac{u}{\varphi}>0$. The Dirichlet boundary condition  $u=0$ on $\Sigma$ ensures that the supremum is not attained on the boundary.  As in \cite[Theorem 3.5]{Graham:1991aa}, we apply Yau's maximum principle so that there is a sequence of points $\{ x_k\}$ such that $\left(\frac{u}{\varphi} \right)(x_k) \to \sup_M \frac{u}{\varphi}$, $\left|\nabla \left(\frac{u}{\varphi} \right)\right|(x_k)\ \to 0$ as $k\to \infty$, and $\liminf_{k\to \infty}\Delta_g \left(\frac{u}{\varphi} \right) (x_k)\le 0$. Taking $\liminf$ on \eqref{eq:diff} along the sequence of points, we then get
\[
	0\ge \liminf_{k\to \infty}  \frac{\mathsf{T} u(x_k)}{\varphi(x_k)} + \epsilon\left(\sup_M \frac{u}{\varphi}\right),
\]
where we use $-\frac{\mathsf{T} \varphi}{\varphi} >\epsilon $. Rearranging the inequality, we can see that it implies the basic estimate. From there, it is standard to show isomorphism from the basic estimate (see \cite[Proposition 3.7]{Graham:1991aa}). This complete the proof for Item~\eqref{item:Dirichlet}.

We now see how to expand the above argument for Item~\eqref{item:Neumann}. The only difference is that  $\sup_M \frac{u}{\varphi}$ can be attained at a boundary point $p\in \Sigma$. Since $\varphi$ is constant near $p$, when we evaluate \eqref{eq:diff} at~$p$,  the  term involving $\nabla \log \varphi$ vanishes. Equation \eqref{eq:diff} at $p$ becomes 
\begin{align*}
	\Delta_g \left(\frac{u}{\varphi} \right)(p)= \frac{\mathsf{T} u}{\varphi} -\left( \sup_M \frac{u}{\varphi}\right) \left(\frac{\mathsf{T} \varphi }{\varphi} \right).
\end{align*}
To derive the basic estimate as above, it suffices to show that $\Delta_g \left(\frac{u}{\varphi} \right)(p)\le 0$. Since $\varphi$ is constant near $p$,  the function $u$ itself also reaches a local maximum at $p$. The condition $\nu(u)=0$ implies $\Delta_g u = \Delta_\Sigma u + \nu(\nu(u)) + \nu(u) H_g =  \Delta_\Sigma u + \nu(\nu(u))$ (we may without loss of generality assume $\nabla_\nu \nu =0$ in a collar neighborhood of $\Sigma$). At the local maximum point $p$, we have $\Delta_\Sigma u(p)\le 0$ and $\nu(\nu(u))(p)\le 0$,  and thus $\Delta_g \left(\frac{u}{\varphi} \right)(p)= \frac{1}{\varphi(p)} \Delta_g u(p)  \le 0$. The rest of the argument is as above. 

For the last statement in the lemma, given $\eta\in \C^{2,\alpha}(\Sigma)$ and $f\in \C^{0,\alpha}_{-s}(M)$, we extend $\eta$ to be a compactly supported $\C^{2,\alpha}$-function in $M$, still denoted by $\eta$. Let $u\in \C^{2,\alpha}_{-s}(M)$ with $u=0$ on $\Sigma$ solve $\mathsf{T} u = -\mathsf{T} \eta + f$. The desired solution $u_1$ is obtained by setting $u_1 = u+\eta$. We can find $u_2$ by a similar argument.
\end{proof}

The following lemma deals with the asymptotics of a static potential and, more generally, the asymptotics of solutions to an \emph{inhomogeneous} static equation $L_g^* V=\tau$. An unbounded open subset of $M\setminus K$ is called a \emph{cone} if it can be expressed as $ [r_*,\infty)\times W$ for some $r_* \gg 1$ and $W$ an open subset of $N$. The proof of the lemma follows closely the argument in \cite[Section 3]{Huang:2020tm}  for asymptotically hyperbolic manifolds without boundary. For the last statement in the following lemma about $\tau=0$, the more precise asymptotics of $V$ can be obtained in \cite[Appendix A]{Chrusciel:2018aa}, in which one can also find an alternative argument for the conclusion that $V$ is identically zero.  

\begin{lemma}[Cf. {\cite{Huang:2020tm, Chrusciel:2018aa}}]\label{lemma:linear}
Let $(M, g)$ be ALH with possibly nonempty boundary.  Let $V\in \C^2_{\mathrm{loc}} (M)$ solve 
\[
L_g^* V=\tau
\]
where $\tau\in \C^0_{1-q}(M)$. Then either there is a cone $U$ such that $V$ grows linearly in $U$, i.e., there is a positive constant $C$ such that $C^{-1} |x| \le |V(x)| \le C |x|$ for all $x\in U$ where $|x| = r(x)$, or $V= O(r^{-d})$ for some $d>0$. 

Furthermore, if $\tau=0$, then either $V$ grows linearly in a cone or $V$ is identically zero in $M$.
\end{lemma}
\begin{remark}\label{remark:dual}
This lemma and the key coercivity estimate (Proposition~\ref{prop:exterior estimate} below) give the main motivation behind the definition of the weight function $\rho$ so that  $\rho = r^{-2q+n-\delta}$ near infinity. It is direct to check that if $V$ grows linearly in a cone $U$, then $V\not \in\big(L^2_{\rho^{-1}}(M)\big)^*=L^2_\rho(M)$: For any number $a>0$, let $f\ge 0$ be a smooth function supported in $U$ such that $f=  r^{-q-a}$ in a smaller cone $U_0$, and then $f \in L^2_{\rho^{-1}}(M)$. If we further choose $a<n-q$ (recall $q<n)$, then 
\[
	\int_M V f \, d\mu_g \ge C \int_{U_0} r^{1-q-a}  r^{n-2} dr d\sigma_g =  C \int_{U_0} r^{n-q-a-1}  dr d\sigma_g=\infty
\]
where we use that $d\mu_g = O(r^{n-2}) \, dr d\sigma_g $. 

Of course one can choose  $f$ as above in $\C^{0,\alpha}_{-q}(M)$. Then the same  computation shows that if $V$ grows linearly in a cone, then  $V\not\in \big( \C^{0,\alpha}_{-q}(M)\big)^*$.
\end{remark}

\begin{proof}
The proof of the first statement follows by considering the asymptotics of $V(\gamma(t))$ where $\gamma(t)$ is a radial geodesic emitted from  $S_r$ for some $r$ large. The equation $L_g^* V=\tau$ is reduced to an inhomogeneous ODE for $V(\gamma(t))$.  One can verify that the same ODE argument in Section 3 of \cite{Huang:2020tm} for complete asymptotically hyperbolic manifolds  does extend for ALH manifolds, as the ODE argument is indifferent to the geometry and topology of $N$.  For the case $\tau=0$,  the maximum principle was used in \cite[Corollary 3.6]{Huang:2020tm} to conclude $V=0$ if $V$ decays to zero at infinity, but that argument doesn't apply in our case of manifolds with boundary. 

Nevertheless, we show how to expand the above argument and derive the same conclusion for the case $\tau=0$. From the first part of the lemma,  we just need to show that if $V(\gamma(t))$ decays to zero at infinity along every radial geodesic $\gamma(t)$, then $V$ is identically zero: The static equation implies that the restriction $V(\gamma(t))$ along any radial geodesic satisfies a homogeneous ODE 
\[
	\frac{d^2}{dt^2} V(\gamma(t))  = (1+ Q(t) ) V(\gamma(t)),
\] 
with $|Q(t)|\le Ce^{-qt}$ because $|\gamma(t)|$ is comparable to $e^t$ (see Equation (3.13) in \cite{Huang:2020tm}). Since we assume that $V(\gamma(t))$ decays to zero at infinity,  Lemma 3.3 of \cite{Huang:2020tm} implies that either (1) $V(\gamma(t))$ is identically zero, or  (2) $C^{-1}e^{-t} \le V(\gamma(t))\le  Ce^{-t}$ for some positive constant $C$, or equivalently $C^{-1}|x|^{-1} \le V(x) \le C |x|^{-1}$. On the other hand, the static equation implies $\Delta_g V - n V=0$, and thus by Lemma~\ref{lemma:isomorphism}, $V\in \C^{2,\alpha}_{-s}(M)$ for any $s<n$. It implies that $V(x)$ must decay at a faster rate toward infinity,  excluding Case (2). We can then conclude that $V$ is identically zero.
\end{proof}

\subsection{The Wang-Chru\'sciel-Herzlich mass integrals}

Recall the definition of a static ALH manifold in \eqref{eq:static} or \eqref{eq:static-system}. When the reference manifold $(\M, \mb)$ is static, one can define mass integrals for an ALH manifold $(M, g)$ that essentially measure the deficit of $(M, g)$ from $(\M, \mb)$ at infinity.  A definition is first given by X. Wang \cite{Wang:2001ww} for conformally compact, asymptotically hyperbolic manifolds with certain asymptotics. The definition below is by  Chru\'sciel-Herzlich~ \cite{Chrusciel:2003aa}. 

\begin{definition} [Wang-Chru\'sciel-Herzlich mass integrals] \label{def:massfunctional}
Let $(\M, \mb)$ be a static, reference manifold  with a static potential $V$. Let $(M,g)$ be ALH with respect to $(\M, \mb)$.  We define
\begin{align}\label{eq:mass2}
	m(g, V) = \lim_{r\rightarrow\infty}\int_{S_r} \left[V\big((\mathrm{div}\, e) (\nu) -\nu (\tr  \, e) \big)+ (\tr  \,e )  \nu (V) - e(\nabla V, \nu)\right] \, d\sigma
\end{align}
where $e = g-\mb$; the outward unit normal $\nu$ to $S_r$, and $\mathrm{div}, \tr , \nabla,  d\sigma$ are all with respect to $\mb$. 
The integral has an alternative formula (see, \cite[Theorem 3.3]{Herzlich:2016aa} and \cite[Equation (4.40)]{Barzegar:2017uh}):
\[
	-\tfrac{n-2}{2} m(g, V)= \lim_{r\to \infty} \int_{S_r} (\Ric_g + (n-1)g)(\nabla V,  \nu) \, d\sigma.
\]
\end{definition}

We note the fundamental fact that the limit in the above definition does converge on an ALH manifold. See \cite[Proposition 2.2]{Chrusciel:2003aa}. Since we will use some of the computations later in Lemma~\ref{lemma:functional}, we include the proof below. In fact, the proof includes an explicit formula for the boundary integral on $\Sigma$ for  Lemma~\ref{lemma:functional}, which is not needed in proving the following lemma. 
\begin{lemma}[\cite{Chrusciel:2003aa}]\label{lemma:massfinite}
Let $(\M, \mb)$ be a static, reference manifold with a static potential $V_0$. Let $(M, \gamma)$ be ALH with respect to $(\M, \mb)$ (possibly with nonempty boundary $\Sigma$) and $V$ be a function on $M$ such that $V-V_0\in \C^2_{1-q}(M)$.  Then $m(\gamma, V)<\infty$. 
\end{lemma}
\begin{proof}
We extend $\mb$ in $M\setminus K$  to a smooth Riemannian metric everywhere in $M$. We fix an arbitrary ALH metric $g$ with respect to $(\M, \mb)$.  In this proof, we can simply let $g= \mb$; this generality of $g$ is used later in Lemma~\ref{lemma:functional}.
Let $e=\gamma-\mb$. We apply integration by parts (twice) and get 
\begin{align*}
	\int_{M} V L_g e\, d\mu_g &= \int_{M} e\cdot L_g^* V \, d\mu_g\\
	&\quad  +\lim_{r\to \infty} \int_{S_r} \left[V\big((\mathrm{div}\, e) (\nu) -\nu (\tr  \, e) \big)+ (\tr  \,e )\,  \nu (V) - e(\nabla V, \nu)\right] \, d\sigma_g\\
	&\quad - \int_{\Sigma} \left[V\big((\mathrm{div}\, e) (\nu) -\nu (\tr  \, e) \big)+ (\tr  \,e )  \nu (V) -e(\nabla V, \nu)\right] \, d\sigma_g,
\end{align*}
where  the unit normal $\nu$ (pointing to infinity), and $\mathrm{div}, \tr , \nabla$ are all with respect to $g$. The boundary integral on $S_r$ limits to $m(\gamma, V) $ as those terms from the difference of $g$ and $\mb$ vanish in the limit. Re-arranging the above integrals, we get 
\begin{align}\label{eq:mass}
\begin{split}
	m(\gamma, V) &= \int_{M} V L_g e\, d\mu_g - \int_{M} e\cdot L_g^* V \, d\mu_g\\
	&\quad  + \int_{\Sigma} \left[V\big((\mathrm{div}\, e)(\nu) -\nu(\tr  \, e) \big)+ (\tr  \,e ) \, \nu (V) - e(\nabla V, \nu)\right] \, d\sigma_g.
\end{split}
\end{align}
For the first volume integral, by Taylor expansion we have 
\begin{align*}
 L_g e&= L_g (\gamma - g) - L_g (\mb - g) \\
&= R_\gamma - R_g + |\nabla (\gamma -g ) |^2 -  (R_{\mb} - R_g )+  O(|\nabla (\mb - g)|^2\\
& = R_\gamma+n(n-1) + O(r^{-2q}).
\end{align*}
By Lemma~\ref{lemma:linear}, we know that $V_0$ grows at most linearly, so does $V$. Thus, the assumption that $r(R_\gamma+n(n-1))\in L^1$ implies $VL_g e$ is also  integrable. That is, the first volume integral converges. The second integral  also converges because $L_g^* V = L_{\mb}^* V_0+ O(r^{1-q}) = O(r^{1-q}) $. This completes the proof.
\end{proof}




\section{Deform scalar curvature and fix boundary geometry }\label{sec:scalar1}

In this section, we prove  Theorem~\ref{thm:surjectivity0}, which is recalled below.  To see how this theorem is used to prove other main results, skip to the next section. 

\begin{manualtheorem}{\ref{thm:surjectivity0}}
Let $(M, g)$ be an ALH manifold (at rate $q$). Given any scalar function $f\in \mathcal B^{0,\alpha}(M)$, there exists a symmetric $(0,2)$-tensor $h\in \mathcal B^{2,\alpha}(M)$ solving $L_g h =f$. 

Furthermore, there exists $\epsilon>0$ and an open subset $\mathcal{U}\subset g+ \mathcal{B}^{2,\alpha}(M)$ such that given any scalar function $f$ with $\| f \|_{\mathcal{B}^{0,\alpha}(M)} < \epsilon$, there exists $\gamma \in \mathcal{U}$ such that $R_{\gamma} = R_g + f$ in $M$.
\end{manualtheorem}

The majority of the proof is to solve the linearized equation $L_g h=f$. Our strategy of proof follows the approach of Corvino~\cite{Corvino:2000td} in which \emph{compact} manifolds with boundary are considered, with the main difference: we do \emph{not} assume that $L_g^*$ has the trivial kernel because we deal with an ALH end and the key coercivity estimate can be derived without that assumption, see Proposition~\ref{prop:coercivity} below.



In the next two propositions, we derive the coercivity estimate. The first proposition concerns only an exterior region, but it is the key estimate. 
\begin{proposition}[Coercivity estimate on an exterior region]\label{prop:exterior estimate}
	Let  $(M, g)$ be ALH with respect to a reference metric $\mb$. There exists constants $R_0, C>0$ such that for all $R>R_0$ and for any function $u\in \C^\infty_c(M)$ (not necessarily vanishing on the boundary $S_R$), the following estimate holds 
	\[
		\|u\|_{H^2_{\rho}(M\setminus B_R)}\le C\|L^{*}_{g}u\|_{L^{2}_{\rho}(M\setminus B_R)}
	\]
	where  $C$ depends on $n,\delta$. 
\end{proposition}

\begin{proof}
It suffices to obtain the desired estimate for $g$ identical to the reference metric $ \mb = \frac{1}{r^2+k} dr^2 + r^2 h$ on $M \setminus B_R$. Once it is obtained for the reference metric, the estimate for a general ALH metric $g$ can be derived automatically (by enlarging $R_0$ if necessary) since the error terms from the difference of $g$ and $\mb$ are negligible. In fact, we can further assume $k=0$ since the difference from the presence $k$ is also negligible. For the rest of the proof, all the geometric quantities (covariant derivatives, volume forms, \dots) are computed with respect to $\mb$.  

Define the differential operator
\[
	T u= \nabla^2 u - u \mb.
\]
We can rewrite  $Tu= L_{\mb}^* u - \frac{1}{n-1} (\tr \, L_{\mb}^* u) \mb + (-\Ric_{\mb} + \frac{1}{n-1} R_{\mb} \mb + \mb)u$. Since $-\Ric_{\mb} + \frac{1}{n-1} R_{\mb} \mb + \mb$ goes to zero at infinity,  we just need to show that there exist constants $R_0, C>0$ such that for all $R> R_0$,
\[
		\|u\|_{H^2_{\rho}(M\setminus B_R)}\le C\|Tu\|_{L^{2}_{\rho}(M\setminus B_R)}.
\]
Furthermore, because $\|\nabla^2 u\|_{L^2_\rho(M\setminus B_R)}$ can be bounded from above by $\|Tu\|_{L^{2}_{\rho}(M\setminus B_R)}$ and a constant multiple of $\| u \|_{L^{2}_{\rho}(M\setminus B_R)}$, it suffices to derive the following $H^1$-estimate:
\[
	\|u\|_{H^1_{\rho}(M\setminus B_R)}\le C\|Tu\|_{L^{2}_{\rho}(M\setminus B_R)}.
\] 
Recall the definition of $H^2_{\rho}(M\setminus B_R)$ in \eqref{eq:norm}. We can re-express the desired estimate as 
\begin{align}
 \label{eq:coercivity}
	\int_{M\setminus B_R} ( u^2 +|\nabla u|^2)r^{-2q+n-\delta}\, d\mu \le C\int_{M\setminus B_R} | Tu|^2 \, r^{-2q+n-\delta} \, d\mu. 
\end{align}
Set the exponent $a = -2q+n-\delta$. Then the range for $q\in (\frac{n}{2}, n)$ implies that 
\begin{align}\label{eq:a}
	-n-\delta <a<-\delta.
\end{align}

We proceed to prove \eqref{eq:coercivity}.  We denote $\nu = \tfrac{\nabla r}{r} = r\partial r$ to be the unit radial vector. Note that $\nu(r^a) = ar^a, \nabla r^a = a r^a\nu$, and $\Delta r^a = a(a-1+n) r^a $. As preparation, we compute the following boundary integrals using integration by parts or Green's formula (note the $\nu$ on $S_R$ is the unit normal pointing to infinity): 
\begin{align*}
	\int_{S_R} u^2 \tfrac{1}{a}\nu(r^a) \, d\sigma&= - \int_{M\setminus B_R} \tfrac{1}{a} u^2 \Delta r^a \, d\mu- \int_{M\setminus B_R} \nabla u^2 \cdot \tfrac{1}{a} \nabla r^a\, d\mu\\
	&= \int_{M\setminus B_R} \Big( -u^2 (a-1+n)  - 2u \nu(u)  \Big)r^a\, d\mu\\
	\int_{S_R} \Big(u^2 \nu(r^a) - \nu(u^2) r^a\Big)\, d\sigma&= -\int_{M\setminus B_R} \Big(u^2 \Delta r^a - r^a \Delta u^2\Big)\, d\mu\\
	&=\int_{M\setminus B_R} \Big( \big (- a(a-1+n) +2n\big) u^2  + 2 |\nabla u|^2  + 2u \tr  (Tu) \Big)r^a\, d\mu\\
	\int_{S_R} \Big( |\nabla u|^2 -u^2 \Big) \tfrac{1}{a} \nu(r^a)\, d\sigma& = -\int_{M\setminus B_R} \Big(|\nabla u|^2 - u^2 \Big) \tfrac{1}{a} \Delta r^a\, d\mu - \int_{M\setminus B_R}\nabla (|\nabla u|^2 - u^2)\cdot \tfrac{1}{a}  \nabla r^a\, d\mu\\
	&=\int_{M\setminus B_R}\bigg(\Big(u^2-|\nabla u|^2  \Big)(a-1+n)  - 2\, (Tu)(\nabla u, \nu)  \bigg) r^a\, d\mu.
\end{align*}
Let $\beta$ be a real number depending on $a$ to be determined. Our estimate is based on this inequality
\begin{align}
	0&\le \int_{S_R} (\beta u - \nu(u))^2r^a\, d\sigma = \int_{S_R} \Big(\beta^2 u^2 - \beta \nu(u^2) + (\nu(u))^2 \Big)r^a \,d\sigma\notag\\
	&\le \int_{S_R} \Big((\beta^2+1)u^2  - \beta \nu(u^2)  + \big(|\nabla u|^2 - u^2\big)\Big) r^a\,d\sigma \notag \quad \quad \mbox{(using that $|\nu(u)|\le |\nabla u|$)}\\
	&=\int_{S_R} \bigg(\Big(\beta^2+1 - a\beta\Big) u^2 r^a + \beta \Big(u^2 \nu(r^a) - \nu(u^2) r^a\Big) + \Big(|\nabla u|^2 - u^2\Big)r^a\bigg)\, d\sigma. \label{eq:basic}
\end{align}
Since the polynomial $\beta^2  +1-a\beta= \left(\beta - \tfrac{a}{2} \right)^2 + 1-\frac{a^2}{4}\ge 0$ provided $ -2 \le a \le 2$. Together with the assumption \eqref{eq:a} for $a$, we separate the cases into   $-2 \le a<-\delta$ and $-n-\delta<a<-2$.

\vspace{10pt}
\noindent {\bf Case 1:}  $-2 \le a< -\delta$. Note that we have $\beta^2 + 1 -a\beta \ge 0$ in this case. Use \eqref{eq:basic}, substituting $r^a = \frac{1}{a} \nu (r^a)$, and  replace the boundary integrals on $S_R$ by our previous computations:
\begin{align*}
	0&\le\int_{S_R} \Big(\beta^2+1 - a\beta\Big) u^2 \tfrac{1}{a}\nu(r^a) + \beta \Big(u^2 \nu(r^a) - \nu(u^2) r^a\Big) + \Big(|\nabla u|^2 - u^2\Big)\tfrac{1}{a} \nu(r^a) \, d\sigma\\
	&= \int_{M\setminus B_R}\Big(  \beta\big(2n-\beta(a-1+n)\big) u^2 + \big( 2\beta - (a-1+n) \big) |\nabla u|^2 \Big)r^a \, d\mu\\
	&\quad + \int_{M\setminus B_R}\Big( -2 (\beta^2 + 1 - a\beta )  u\nu(u) + \big( 2\beta u  \tr \, (Tu)   - 2\, (Tu) (\nabla u, \nu) \big) \Big) r^a\, d\mu \\
	&\le \int_{M\setminus B_R}  \Big(\beta \big(2n-\beta(a-1+n)\big) + \beta^2 + 1 -a\beta \Big) u^2 r^a\, d\mu \\
	&\quad + \int_{M\setminus B_R} \Big( 2\beta - (a-1+n)+ \beta^2 + 1 -a\beta \Big) |\nabla u|^2 r^a \, d\mu\\
	&\quad  + \int_{M\setminus B_R} \Big( 2 \beta u  \tr  (Tu)   - 2\, (Tu)(\nabla u, \nu) \Big)r^a\, d\mu
\end{align*}
where we use the Cauchy-Schwarz inequality for the term $-2 u\nu(u) \le |u|^2 + |\nabla u|^2$ and that $\beta^2+1-a\beta\ge 0$. We denote the ``coefficients'' of $u^2 r^a$ and $|\nabla u|^2 r^a$ by, respectively:
\begin{align*}
	c_1(a, \beta) &:=  \beta\big(2n-\beta(a-1+n) \big) + \beta^2  +1-a\beta\\
	c_2(a, \beta)&:= 2\beta - (a-1+n)+ \beta^2+1-a\beta.
\end{align*}
Let $\beta=\frac{a}{2}$ so that $ \beta^2  -a\beta+1$ takes the minimum value $1-\tfrac{a^2}{4}$. We verify that $c_1, c_2$ are negative constants depending on $n, \delta$:   \begin{align*}
 	c_1(a, \tfrac{a}{2})&= \tfrac{a}{2} \left(2n - \tfrac{a}{2} (a-1+n)\right) + 1 - \tfrac{a^2}{4}\\
	 &= an - \tfrac{a^2}{4} (a+n) + 1\\
	 &\le an - \tfrac{a^2}{4} (n-2) + 1\\
 &	= - \tfrac{n-2}{4} \left( a - \tfrac{2n}{n-2}\right)^2  + \tfrac{n^2}{n-2} +1.
 \end{align*}
The right hand side, as a polynomial of $a\in [-2, -\delta)$, has an upper bound at $a=-\delta$. Evaluating the polynomial at $a=-\delta$, one can verify that $- \frac{n-2}{4} \left( -\delta - \frac{2n}{n-2}\right)^2  + \frac{n^2}{n-2} +1<0$, provided that $\delta$ stays in the  range~\eqref{eq:d}. The estimate for $c_2$ is obvious since we  have
 \begin{align*}
 	c_2(a, \tfrac{a}{2})&= a - (a-1+n) + 1 - \tfrac{a^2}{4}= 2-n - \tfrac{a^2}{4}\le 2-n.  
 \end{align*}

\vspace{10pt}
\noindent {\bf Case 2:} $-n-\delta< a< -2$. We will consider $\beta$ so that $\beta^2  -a\beta+1\le 0$ (e.g., we will set $\beta=-1$ shortly), and  thus the first integrand in \eqref{eq:basic} can be dropped: 
\begin{align*}
0&\le \int_{S_R} \beta \Big(u^2 \nu(r^a) - \nu(u^2) r^a\Big) + \Big(|\nabla u|^2 - u^2\Big)r^a\, d\sigma\\
	&= \int_{M\setminus B_R}  \Big( \beta \big(-a(a-1+n) + 2n\big) + (a-1+n) \Big) u^2 r^a\, d\mu\\
	&\quad + \int_{M\setminus B_R} \Big( 2\beta - (a-1+n) \Big) |\nabla u|^2 r^a \, d\mu+ \int_{M\setminus B_R} \Big( 2 \beta u  \tr (Tu)   - 2\, (Tu)(\nabla u, \nu) \Big)r^a \, d\mu.
\end{align*}
We denote the ``coefficients'' of $u^2 r^a$ and $|\nabla u|^2 r^a$ by, respectively,
\begin{align*}
	c_3(a, \beta) &:=  \beta \big(-a(a-1+n) + 2n\big) + (a-1+n)\\
	c_4(a, \beta)&:= 2\beta - (a-1+n).
\end{align*}
Let $\beta=-1$. Then 
\begin{align*}
	c_3(a, -1) &:=  -(-a(a-1+n) + 2n) + (a-1+n)\\
	&=(a+1)(a-1+n)-2n\\
	c_4(a,- 1)&:= -2 - (a-1+n)\\
	&=-a-n-1.
\end{align*}
One can verify that for  $-n-\delta<a<-2$, both $c_3$ and $c_4$ are negative constants that can depend on $n, \delta$. 

\vspace{8pt}

To complete the proof for both cases, we find a constant $\epsilon:=\epsilon(n,\delta)>0$ such that 
\begin{align*}
0 &\le -2\epsilon \int_{M\setminus B_R}  (u^2 + |\nabla u|^2) r^a \,d\mu + \int_{M\setminus B_R} \Big( 2 \beta u  \tr (Tu)   - 2\, (Tu)(\nabla u, \nu) \Big)r^a \, d\mu.
\end{align*}
Noticing that by Cauchy-Schwarz, there exists a positive constant $C(\epsilon)$ such that 
\begin{align*}
	2 \beta u ( \tr  (Tu) ) & \le \epsilon u^2 + C(\epsilon) |Tu|^2\\
	 - 2Tu(\nabla u, \nu) &\le \epsilon |\nabla u|^2 +C(\epsilon) |Tu|^2.
\end{align*}
Combining the above inequalities, we obtain the desired estimate~\eqref{eq:coercivity}:
\[
	 \int_{M\setminus B_R}  (u^2 + |\nabla u|^2) r^a\, d\mu\le C \int_{M\setminus B_R} |Tu|^2 r^a\, d\mu. 
\]

\end{proof}

\begin{proposition}[Coercivity estimate]\label{prop:coercivity}
Let $(M, g)$ be ALH. There exists a positive constant $C$ such that for all $u\in H^2_\rho(M)$
\begin{align}\label{eq:Basic estimate on M}
	\| u \|_{H^2_\rho(M)}\le C\| L^*_g u \|_{L^2_\rho(M)}.
\end{align}
\end{proposition}
\begin{proof}
We  \emph{claim} that there exist constants $R_0, C>0$ such that  for $R> R_0$ and for all $u\in H^2_\rho(M)$,  the following estimate holds:
\begin{align}\label{eq:elliptic}
\| u \|_{H^2_\rho(M)}  \le C \left(\| L_g^* u \|_{L^2_\rho(M)}+ \| u \|_{L^2_\rho (B_R)}\right)
\end{align}
We just need to derive the estimate for $u\in \C^\infty_c(M)$. Let $R\ge R_0$ where $R_0$ is from Proposition~\ref{prop:exterior estimate}. By \eqref{eq:exterior} and applying   Proposition~\ref{prop:exterior estimate} in the second line below:
\begin{align*}
	\| u \|_{H^2_\rho(M)} &\le \| u \|_{H^2_\rho(B_R)} + \| u \|_{H^2_\rho(M\setminus B_R)}\\
	&\le \| u \|_{H^2_\rho(B_R)} + C \| L_g^* u \|_{L^2_\rho(M\setminus B_R)}.
\end{align*}	
For the norm in $B_R$, we have the standard elliptic estimate
\[
	 \| u \|_{H^2_\rho(B_R)} \le C \left(\| L_g^* u \|_{L^2_\rho(B_R)} + \| u \|_{L^2_\rho (B_R)}\right).
\]
See \cite[Lemma 5.1]{Corvino:2020ty}, which follows the proof of \cite[Proposition 3.1-3.2, Theorem 3]{Corvino:2000td}. (Technically speaking,  \cite[Lemma 5.1]{Corvino:2020ty} does not directly apply since our weight function $\rho$ does not decay to zero toward the ``outer'' boundary $S_R$, but the estimate toward $S_R$ is a more standard estimate because the weight function $\rho$ is positive there.) Combining the previous two inequalities and enlarging the constant $C$,  we have the following estimate:
\begin{align*}
\| u \|_{H^2_\rho(M)}  \le C \left(\| L_g^* u \|_{L^2_\rho(M)}+ \| u \|_{L^2_\rho (B_R)}\right)
\end{align*}
where we use \eqref{eq:exterior} to combine the integrals of $L_g^* u$ on $B_R$ and on $M\setminus B_R$. 

To prove the theorem, we argue by contradiction. Suppose the estimate \eqref{eq:Basic estimate on M} does not hold for all $u\in H^2_\rho(M)$. Then there exists a sequence $u_i\in H^2_\rho(M)$, which we normalize to make $\| u_i \|_{H^2_\rho(M)}=1$, such that 
\begin{align} \label{eq:limit}
	 \| L_g^* u_i \|_{L^2_\rho(M)} \le \frac{1}{i} \| u_i \|_{H^2_\rho(M)} \to 0 \quad \mbox{ as } i\to \infty.
\end{align}

It is direct to verify that there is $C>0$ such that $\| u_i \rho^{1/2} \|_{H^1(B_R)} \le C\| u_i \|_{H^1_\rho(B_R)}$ (see, e.g. \cite[Proposition 2.10]{Corvino:2020ty}). By the assumption  $ \| u_i \|_{H^2_\rho(M)}=1$ and Rellich compactness theorem (with respect to the usual, unweighted norms), after passing to a subsequence, $u_i \rho^{1/2}$ converges to some $f$ in $L^2(B_R)$. It implies that $\| u_i - f \rho^{-1/2} \|_{L^2_\rho(B_R)} \to 0$  as $i\to \infty$. We can now use \eqref{eq:elliptic} to show that the sequence $u_i$ is a Cauchy sequence in $H^2_\rho (M)$:
\begin{align*}
	\| u_i - u_j \|_{H^2_\rho(M)}& \le C \left(\| L_g^* (u_i - u_j) \|_{L^2_\rho(M)}+ \| u_i -u_j \|_{L^2_\rho (B_R)}\right)\\
	&\le C \left(\| L_g^* u_i\|_{L^2_\rho(M)} + \| L_g^* u_j \|_{L^2_\rho(M)}+\| u_i -u_j \|_{L^2_\rho (B_R)}\right) \to 0 \quad \mbox{ as } i, j\to \infty.
\end{align*}
 Denote by $u\in H^2_\rho(M)$ the limit of the sequence $u_i$. So we have that $\| u \|_{H^2_\rho(M)}=1$ and $u$ is a weak solution to $L_g^* u=0$ by \eqref{eq:limit}. By elliptic regularity, $u\in \C^2_{\mathrm{loc}}(M)$ and by Lemma~\ref{lemma:linear}, either $u$ is identically zero or $u$ grows linearly in a cone. The latter cannot occur since $u\in H^2_\rho(M)$ (see Remark~\ref{remark:dual}), so $u$ must vanish identically, but it contradicts that $\| u \|_{H^2_\rho(M)}=1$.
\end{proof}

Once the coercivity estimate is derived, we can use the variational argument as in \cite{Corvino:2000td} for any given $f$ to construct a weak solution $h$ to $L_g h=f$.  Given $f\in L^2_{\rho^{-1}}(M)$, we define the functional $\mathcal{G}: H^2_\rho(M)\to \mathbb{R}$ by
\[
		\mathcal{G}(u) = \int_M \left(\tfrac{1}{2} \rho |L_g^* u|^2 - fu \right)\, d\mu_g.
\]
\begin{theorem}\label{thm:weak-solution}
Let $(M, g)$ be ALH. For any $f\in L^2_{\rho^{-1}}(M)$, $\mathcal{G}$ has a global minimizer $u\in H^2_\rho(M)$, and $u$ weakly solves
\[
	L_g ( \rho L_g^* u) = f\quad \mbox { in M}
\]
with  the estimate 
\begin{align}\label{eq:u}
\| u \|_{H^2_\rho(M)}\le 2C \| f \|_{L^2_{\rho^{-1}}(M)}
\end{align}
 where $C$ is the constant from Proposition~\ref{prop:coercivity}.  
\end{theorem}
\begin{proof}
We use Proposition~\ref{prop:coercivity} in the second line below:
\begin{align*}
	\mathcal{G}(u) &\ge  \tfrac{1}{2} \| L_g^* u\|_{L^2_\rho (M)}^2 - \| f \|_{L^2_{\rho^{-1}}(M)} \| u \|_{L^2_\rho(M)}\\
	&\ge \tfrac{1}{2C} \| u \|_{H^2_\rho(M)}^2 - \| f \|_{L^2_{\rho^{-1}}(M)} \| u \|_{L^2_\rho(M)}.
\end{align*}
One can argue as in \cite[pp. 150-152]{Corvino:2000td} that the infimum of $\mathcal{G}$ is negative and bounded from below.  By taking a minimizing sequence, there is a unique global minimizer $u$ and $u$ solves the Euler-Lagrange equation $L_g ( \rho L_g^* u) = f$. The estimate for $u$ follows that $\mathcal{G}(u)\le 0$.
\end{proof}

Since we are interested in solving $L_g (\rho L_g^* u) =f$ when the source term $f$ has  better regularity, i.e. $f\in \mathcal{B}^{0,\alpha}(M)$, we shall see how to use elliptic regularity to show that the weak solution $u$ has better regularity.  To be more specific, we can obtain that $u\in \mathcal B^{4,\alpha}(M)$, where  
\[
	\mathcal B^{4, \alpha}(M):= \C^{4,\alpha}_{\phi, \phi^{\frac{n}{2}}\rho^{\frac{1}{2}}}(M)\cap H^2_\rho(M)\cap \C^{4,\alpha}_{q-n+\delta}(M).
\]
(cf. the definition of $\mathcal B_4(\Omega)$ in \cite[p. 42]{Corvino:2020ty})  It then implies that $h:= \rho L_g^* u\in \mathcal B^{2,\alpha}(M)$. (Recall $\rho(x)= r^{-2q+n-\delta}$ outside a compact set, so $v\in \C^{\ell,\alpha}_{-q}(M)$ if and only if  $\rho^{-1} v\in \C^{\ell,\alpha}_{q-n+\delta}(M)$.)

By the interior Schauder estimate, we know already $u\in \C^{4, \alpha}_{\mathrm{loc}}(M)$. Since the weight function $\rho$ in the equation for $L_g ( \rho L_g^* u) $ results in degeneracy near $\Sigma$ and at infinity, so we consider the following differential operator 
\[
	P u := \rho^{-1} L_g (\rho L_g^* u) = \rho^{-1} f.
\]
One can verify that the leading order terms of $P$ is of the form $L_g L_g^* u$  and hence is uniformly elliptic near $\Sigma$. Therefore,  \cite[Theorem 5.6]{Corvino:2020ty} takes care of the estimates near $\Sigma$, and we conclude that $u\in \mathcal B^{4,\alpha}(\Omega)$ and $\rho L_g^* u\in \mathcal{B}^{2,\alpha} (\Omega)$ for any compact subset $\Omega\subset M$.  So our main task is to estimate outside a compact subset. In fact, it suffices to show that $u\in \C^{4,\alpha}_{q-n+\delta}(M)$, the norm that is ``unweighted'' near the boundary $\Sigma$.  (One should readily check that $\C^{4,\alpha}_{\phi, \phi^{\frac{n}{2}}\rho^{\frac{1}{2}}}(M)$ and $H^2_\rho(M)$ in the definition of $\mathcal B^{4,\alpha}$ impose less restricted asymptotics at infinity).  Near infinity the differential operator $P$ becomes a $4$th order, \emph{uniformly degenerate} operator in the sense of \cite{Graham:1991aa}. We explain how to apply their estimate to our case.  


\begin{proposition}\label{prop:regularity}
Let $f \in \mathcal{B}^{0,\alpha}(M)$. There is a constant $C>0$ such that for any $u\in L^2_\rho(M)$ weakly solving $L_g ( \rho L_g^* u)=f$, we have  $u\in \C^{4,\alpha}_{q-n+\delta}(M)$ and 
\begin{align*}
	\| u \|_{\C^{4,\alpha}_{q-n+\delta}(M)}&\le C \left(\| \rho^{-1}f\|_{\C^{0,\alpha}_{q-n+\delta}(M)} + \| u \|_{L^2_{\rho}(M)}\right).
\end{align*}

Consequently,  the solution $u$ obtained in Theorem \ref{thm:weak-solution} is in $\mathcal B^{4,\alpha}(M)$, and as a result, $h:= \rho L_g^* u$ is in $\mathcal B^{2,\alpha}(M)$. Together with the estimate \eqref{eq:u} to replace the $L^2_\rho$-norm of $u$ in the right hand side of the above estimate, we obtain
\begin{align}\label{eq:est}
\begin{split}
	\| u \|_{\mathcal B^{4,\alpha}(M)} &\le C \| f\|_{\mathcal B^{0,\alpha}(M)}\\
	\| h \|_{\mathcal B^{2,\alpha}(M)} &\le C \| f\|_{\mathcal B^{0,\alpha}(M)}.
\end{split}
\end{align}

\end{proposition} 
\begin{proof}
It suffices to show the H\"older estimate. The desired estimate restricted in a compact subset of $M$ including $\Sigma$ becomes the  Schauder estimate in the standard (unweighted) H\"older norms, which is implied by the stronger estimate near boundary obtained in \cite[Appendix C]{Corvino:2020ty}. Therefore, it suffices to find compact subsets $\Omega' \subset \subset \Omega$ both containing $B_R$ for $R$ sufficiently large such that the following estimate holds for some $C>0$ (letting $Pu = \rho^{-1} f$):
\begin{align} \label{eq:Holder}
  \| u \|_{\C^{4,\alpha}_{q-n+\delta}(M\setminus \Omega)}\le C \left(\| Pu\|_{\C^{0,\alpha}_{q-n+\delta}(M\setminus \Omega')} + \| u \|_{L^2_{\rho}(M\setminus \Omega')}\right).
\end{align}

We just need to  derive \eqref{eq:Holder} when $g$ is the reference metric $\mb$ because the error terms from the difference from $g$ and $\mb$ can be absorbed into the left hand side for sufficiently large $R$. Using that $\mb$ is conformally compact, the differential operator $P u =\rho^{-1} L_{\mb} (\rho L_{\mb}^* u)$ is a 4th order \emph{uniformly degenerate} operator in the sense of Graham and Lee~\cite{Graham:1991aa}. Therefore,  we can apply the rescaling argument \cite[Proposition 3.4]{Graham:1991aa} to show that, for any real number $s$ and any $x\in B_R$,  there is a constant $C>0$, independent of $x$,  such that 
\begin{align*}
	\| u \|_{\C^{4,\alpha}_{-s}(B_{\frac{1}{2r(x)}}(x))}\le C \left(\| Pu\|_{\C^{0,\alpha}_{-s}(B_{\frac{1}{r(x)}}(x))} + \| u \|_{L^2_{r^{-n+2s}}(B_{\frac{1}{r(x)}}(x))}\right).
\end{align*}
Note that the $\C^0$-norm in the original estimate in \cite{Graham:1991aa} can be replaced by the $L^2$-norm using an interpolation inequality, just as in \cite[Appendix C]{Corvino:2020ty}. (Also note the space $\Lambda^s_{\ell,\alpha}$ used in \cite{Graham:1991aa}  is equivalent to our weighted space $\C^{\ell,\alpha}_{-s}$.) Let $-s= q-n+\delta$ and note $\| u \|_{L^2_{r^{-n+2s}}(B_{\frac{1}{r(x)}}(x))}\le \| u \|_{L^2_{\rho}(B_{\frac{1}{r(x)}}(x))}$ for such $s$. Taking the supremum of the local estimates among $x$ and letting 
\[
	\Omega = M\setminus \bigcup_{x\in B_R } B_{\frac{1}{2r(x)}}(x) \mbox{ and } \Omega' = M\setminus \bigcup_{x\in B_R } B_{\frac{1}{r(x)}}(x),
\]
this completes the proof.

\end{proof}

We combine the above arguments to complete the proof of Theorem~\ref{thm:surjectivity0}.
\begin{proof}[Proof of Theorem~\ref{thm:surjectivity0}]
Given $f\in \mathcal{B}^{0,\alpha}(M)$, by Theorem~\ref{thm:weak-solution}, there is a weak solution $u\in H^2_{\rho}(M)$ to $L_g ( \rho L_g^* u)=f$. We let $h= \rho L_g^* u\in L^2_{\rho^{-1}}(M)$. Then by  Proposition~\ref{prop:regularity}, we see that $h $ is a strong solution and belongs to the desired space $\mathcal{B}^{2 ,\alpha}(M)$. 

To solve for the nonlinear problem, we adapt the iteration scheme of Corvino~\cite{Corvino:2000td} (see also \cite[Theorem 5.10]{Corvino:2020ty}), together with the estimates~\eqref{eq:est}. The argument follows verbatim as \cite[Theorem 5.10]{Corvino:2020ty}), which we outline below. We would like to show that there is $\epsilon>0$ small such that for given $f$ with $\| f \|_{\mathcal B^{0,\alpha}(M)} < \epsilon$, there is $\gamma\in \mathcal B^{4,\alpha}(M)$ solving $R_\gamma = R_g + f$. Let $u_0\in \mathcal B^{4,\alpha}(M)$ be the variational solution from Theorem~\ref{thm:weak-solution} to the following:
\begin{align*}
	L_g (\rho L_g^* u_0 ) &= f\\
	h_0 &:= \rho L_g^* u_0\\
	\gamma_1 &:= g+h_0.
\end{align*} 
From the estimates \eqref{eq:est}, we know that for $\epsilon$ small, $\gamma_1$ is still a Riemannian metric and
\begin{align*}
	\| u_0\|_{\mathcal B^{4,\alpha}(M)} &\le C \| f \|_{\mathcal B^{0,\alpha}(M)}\\
	\| h_0\|_{\mathcal B^{2,\alpha}(M)} &\le C \| f \|_{\mathcal B^{0,\alpha}(M)}\\
	\| R_g + f - R_{\gamma_1} \| &\le C \| f \|_{\mathcal B^{0,\alpha}(M)}^2. 
\end{align*} 
We then proceed recursively and let, for $m=1, 2, \dots,$ 
\begin{align*}
	L_g (\rho L_g^* u_m ) &= R_g+ f - R_{\gamma_m}\\
	h_m &:= \rho L_g^* u_m\\
	\gamma_{m+1}&= g+ \sum_{p=0}^{m} h_p.
\end{align*}
Following \cite[Lemma 5.11]{Corvino:2020ty}, for $\epsilon$ sufficiently small the estimates ensure the series $\sum_{p=0}^\infty u_p$ converges  to some $u$ in $\mathcal B^{4,\alpha}(M)$. Let $h:= \rho L_g^* u$. Then $\gamma = g+h$ satisfies the nonlinear equation $R_\gamma = R_g + f$. 
\end{proof}

\section{Deform scalar curvature and prescribe Bartnik boundary data}\label{sec:scalar2}


For a Riemannian manifold $(M, g)$ with boundary $\Sigma$, the \emph{Bartnik boundary data} on $\Sigma$ is $(g^\intercal, H_g)$ where $g^\intercal$ is the induced metric on $\Sigma$ and $H_g = \mathrm{div}_\Sigma \nu$. (Note that for an ALH manifold $(M, g)$ we fix the unit normal $\nu$ to point to infinity.)   If $h$ is a variation of $g$ in $M$, then the \emph{linearized} Bartnik boundary data is $(h^\intercal, DH|_g(h))$, where  $h^\intercal$ denotes the restriction of $h$ on the tangent bundle of $\Sigma$ and $DH|_g(h)$ is the linearized mean curvature, given by the formula (see, e.g. \cite[Lemma 2.1]{An:2021tw})
\begin{align} \label{eq:mean-curvature}
DH|_g(h)= \tfrac{1}{2} \nu(\tr \, h^\intercal) - \mathrm{div}_\Sigma \omega - \tfrac{1}{2} h(\nu, \nu) H_g
\end{align}
 where $\omega (e_\alpha) = h(\nu, e_\alpha)$ is a one-form on the tangent bundle of $\Sigma$.

Let $(M, g)$ be an ALH manifold at rate $q$. For a $(0, 2)$-tensor $h\in \C^{2,\alpha}_{-q}(M)$, we define the linear operator 
\begin{align} \label{eq:T}
	T(h )= (L_g h, h^\intercal, DH|_g(h)).
\end{align}
Recall that $L_g$ denotes the linearized scalar curvature operator. The main goal in this section is to prove Theorem~\ref{thm:boundary0}. We just need to show that the map $T$ is surjective as follows. 
\begin{theorem}\label{thm:surjectivity-Bartnik}
Let $(M, g)$ be ALH at rate $q$. Then $T:   \C^{2,\alpha}_{-q}(M)\longrightarrow \C^{0,\alpha}_{-q}(M) \times \C^{2,\alpha}(\Sigma)\times \C^{1,\alpha}(\Sigma)$.
is surjective. 
\end{theorem}

Then Theorem~\ref{thm:boundary0} follows from the above theorem: Consider an open neighborhood $\mathcal U$ of $g$ in $g+\C^{2,\alpha}_{-q}(M)$ of Riemannian metrics and the smooth map $F: \mathcal U\longrightarrow \C^{0,\alpha}_{-q}(M) \times \C^{2,\alpha}(\Sigma)\times \C^{1,\alpha}(\Sigma)$ defined by $F(\gamma)=(R_\gamma, \gamma^\intercal, H_\gamma)$. Since the linearization  $DF|_g=T$ is surjective, by Local Surjectivity Theorem, $F$ is locally surjective.

An analogous statement for asymptotically flat manifolds, Theorem~\ref{thm:AF} below, was first obtained by Anderson-Jauregui for $n=3$ in ~\cite[Proposition 2.4]{Anderson:2019tm}. In Appendix~\ref{sec:af} we provide a different proof for general dimensions $n\ge 3$. 
\begin{theorem}[Cf. \cite{Anderson:2019tm}]\label{thm:AF}
Let $(M', g')$ be an $n$-dimensional asymptotically flat with compact boundary $\Sigma$.  Then the map from a symmetric $(0,2)$-tensor $h\in  \C^{2,\alpha}_{-s}(M')$ to $(L_{g'} h, h^\intercal, DH|_{g'}(h)) \in \C^{0,\alpha}_{-s}(M') \times \C^{2,\alpha}(\Sigma)\times \C^{1,\alpha}(\Sigma)$ is surjective. 
\end{theorem}

We will not explicitly use the definitions of asymptotically flat manifolds nor the corresponding weighted H\"older spaces as stated in the above theorem, because we essentially use a ``localized'' version of Theorem~\ref{thm:AF} to prove Theorem~\ref{thm:surjectivity-Bartnik}.  We outline  the approach:
\begin{enumerate}
\item From an ALH manifold $(M, g)$ with boundary $\Sigma$,  we construct an asymptotically flat manifold $(M', g')$ that contains an isometric copy of a collar neighborhood of the boundary $\Sigma$ in $(M, g)$. Since we do not require $g'$ to satisfy any curvature condition, this step is fairly easy.
\item We can then apply a ``localized'' version of Theorem~\ref{thm:AF}:  Given any $(f, \tau, \phi)$, we find a compactly supported, symmetric $(0,2)$-tensor $h_0$  in the collar neighborhood of the boundary $M'$ such that $L_{g'} h_0$ is identical to $f$ in a collar neighborhood of $\Sigma$ and $(h_0^\intercal, DH|_{g'}(h_0)) = (\tau, \phi)$ on $\Sigma$. We then apply Theorem~\ref{thm:surjectivity0} to solve the linearized scalar curvature equation with the source term  $f - L_{g'}h_0$. 
\end{enumerate}

The following lemma is an exact statement of Step (1).

\begin{lemma}\label{lemma:interpolation2}
Let $(M, g)$ be ALH. Given any bounded open subset $U\subset M$, there exists an asymptotically flat manifold $(M', g')$ such that $(M', g')$ contains an isometric copy of $(\overline{U}, g)$.
\end{lemma}
\begin{proof}
We ``chop off'' an exterior region $M\setminus B_R$  for $R$ sufficiently large such that $B_R$ contains $\overline{U}$. Then we perform a connected sum on $B_R\setminus \overline{U}$ with the Euclidean space $(\mathbb{R}^n, g_{\mathbb{E}})$. By smoothly extending the Riemannian metric across the neck region where the connected sum is performed, the resulting manifold $M'= B_R \# \mathbb{R}^n$ with the extending metric is an asymptotically flat manifold that contains a isometric copy of $(\overline{U}, g)$. See Figure~\ref{figure3}.
\end{proof}

\begin{figure}[ht]
	\def\svgwidth{3in}
\begingroup%
  \makeatletter%
  \providecommand\color[2][]{%
    \errmessage{(Inkscape) Color is used for the text in Inkscape, but the package 'color.sty' is not loaded}%
    \renewcommand\color[2][]{}%
  }%
  \providecommand\transparent[1]{%
    \errmessage{(Inkscape) Transparency is used (non-zero) for the text in Inkscape, but the package 'transparent.sty' is not loaded}%
    \renewcommand\transparent[1]{}%
  }%
  \providecommand\rotatebox[2]{#2}%
  \newcommand*\fsize{\dimexpr\f@size pt\relax}%
  \newcommand*\lineheight[1]{\fontsize{\fsize}{#1\fsize}\selectfont}%
  \ifx\svgwidth\undefined%
    \setlength{\unitlength}{523.12171275bp}%
    \ifx\svgscale\undefined%
      \relax%
    \else%
      \setlength{\unitlength}{\unitlength * \real{\svgscale}}%
    \fi%
  \else%
    \setlength{\unitlength}{\svgwidth}%
  \fi%
  \global\let\svgwidth\undefined%
  \global\let\svgscale\undefined%
  \makeatother%
  \begin{picture}(1,0.92559549)%
    \lineheight{1}%
    \setlength\tabcolsep{0pt}%
    \put(0.49140704,0.17876875){\makebox(0,0)[lt]{\lineheight{1.25}\smash{\begin{tabular}[t]{l}$\Sigma$\end{tabular}}}}%
    \put(0,0){\includegraphics[width=\unitlength,page=1]{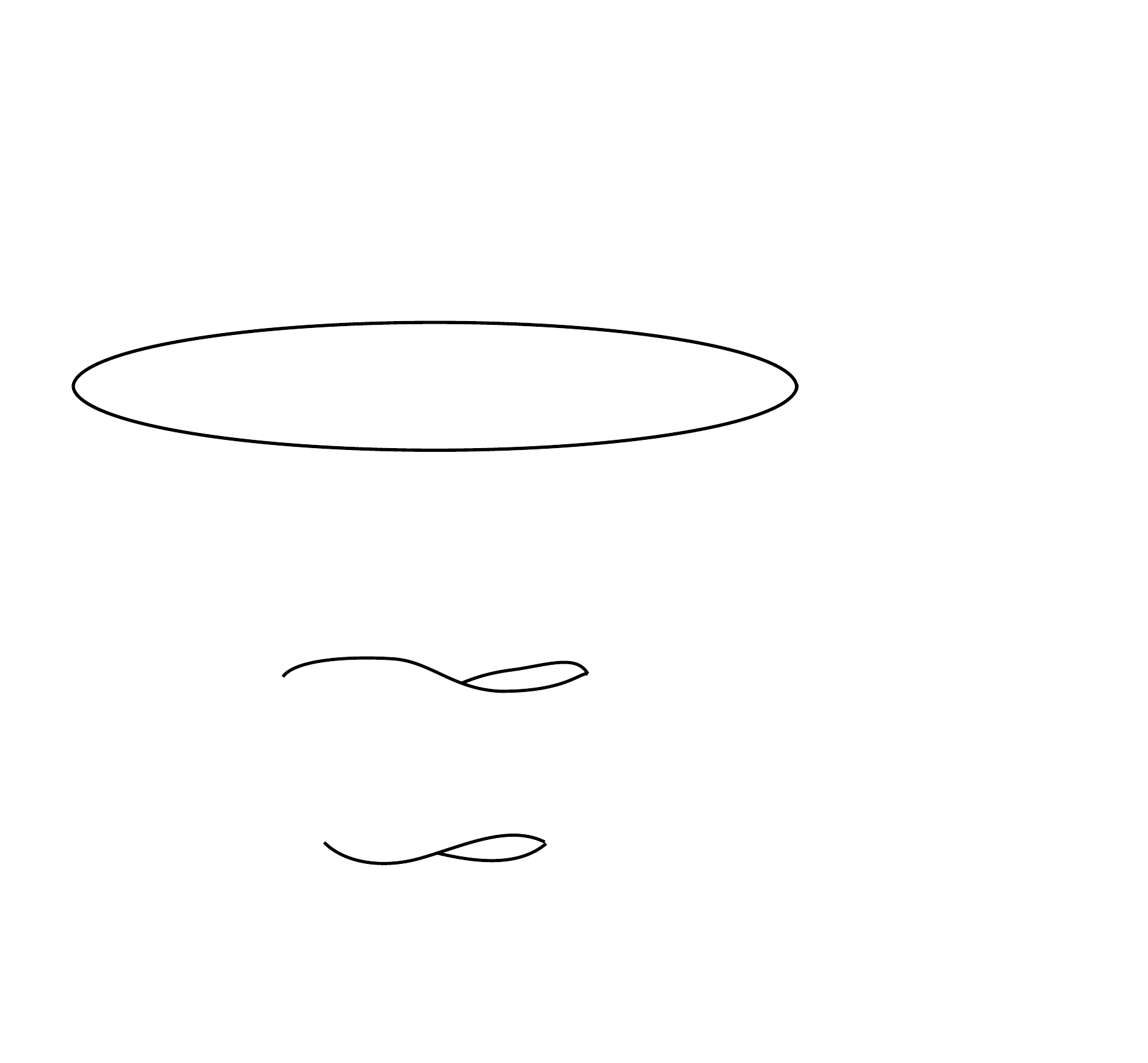}}%
    \put(0.53510867,0.26820039){\makebox(0,0)[lt]{\lineheight{1.25}\smash{\begin{tabular}[t]{l}$U$\end{tabular}}}}%
    \put(0,0){\includegraphics[width=\unitlength,page=2]{fig3_connected_sum_3.pdf}}%
    \put(-0.00289354,0.66766013){\makebox(0,0)[lt]{\lineheight{1.25}\smash{\begin{tabular}[t]{l}$\partial B_R$\end{tabular}}}}%
    \put(0,0){\includegraphics[width=\unitlength,page=3]{fig3_connected_sum_3.pdf}}%
  \end{picture}%
\endgroup%

	\caption{$M'=B_R\# \mathbb{R}^n$ in Lemma~\ref{lemma:interpolation2}}
	\label{figure3}
\end{figure}

In the following proposition, we carry out Step (2) for the zero Bartnik boundary data.
\begin{proposition}\label{prop:surjectivity 2}
Let $(M, g)$ be ALH. Then the linearized scalar curvature $L_g: \big\{ h\in \C^{2,\alpha}_{-q}(M) : h^\intercal =0, DH|_g(h)=0  \mbox{ on } \Sigma\big \}\to \C^{0,\alpha}_{-q}(M)$ is surjective. 
\end{proposition}
\begin{proof}
Given $f\in \C^{0,\alpha}_{-q}(M)$, we will find $h\in \C^{2,\alpha}_{-q}(M)$ with  $h^\intercal =0, DH|_g(h)=0 $ on $\Sigma$ such that $L_g h = f$.  

Fix two nonempty bounded open subsets $U_1, U_2\subset M$  with $\Sigma\subset U_1$ and $\overline{U_1} \subset U_2$. By Lemma~\ref{lemma:interpolation2}, there exists an asymptotically flat manifold $(M', g')$ that contains an isometric copy of $(\overline{U_2}, g)$. We still denote the isometric image in $(M', g')$ by $(\overline{U_2}, g)$ and thus $g'=g$ on $\overline{U_2}$.   Let  $f_0$ be a function supported in $U_1$ such that $f_0=f$ in a collar neighborhood of $\Sigma$. With respect to the asymptotically flat manifold $(M', g')$, Theorem~\ref{thm:AF} says that there exist $h_0$ satisfying 
\begin{align}
	L_{g'} (h_0) &= f_0\quad \mbox{ in } M' \notag\\
	h_0^\intercal&=0,\quad DH|_{g'}(h_0)=0\quad \mbox{ on }\Sigma. \label{eq:Bartnik}
\end{align}
 Since $h_0$ may not be compactly supported, we multiply $h_0$ by a smooth bump function $\eta$ satisfying $\eta  = 1$ in $U_1$  and $\eta = 0$ outside $U_2$. Since $g'=g$ in $\overline{U_2}$, with respect to the ALH metric $g$, we have
\[
	L_{g} (\eta h_0) = \left\{ \begin{array}{ll} f_0 & \mbox{ in } U_1 \\ 0 & \mbox{ outside } U_2 \end{array} \right..
\]	
Therefore, we can view $L_g(\eta h_0)$ as a scalar function on the ALH manifold $(M, g)$ with support in $\overline{U_2}$, which we smoothly extend to be zero outside $\overline{U_2}$. Therefore, $f- L_{g} (\eta h_0)$ vanishes in a collar neighborhood of $\Sigma$, and thus $f- L_{g} (\eta h_0)\in \mathcal{B}^{0,\alpha} (M)$. By Theorem~\ref{thm:surjectivity0}, there exists $h_1\in \mathcal{B}^{2,\alpha}(M)$ such that 
\[
L_g (h_1) = f - L_g(\eta h_0 ).
\]
 That is,  we obtain a solution $h_1+\eta h_0$ solving $L_g (h_1+ \eta h_0) = f$. It is direct to see that $h_1+\eta h_0 \in  \C^{2,\alpha}_{-q}(M)$ and has vanishing linearized Bartnik boundary data. It completes the proof.

\end{proof}

We combine the above results to prove the main theorem in this section, Theorem~\ref{thm:surjectivity-Bartnik}.

\begin{proof}[Proof of Theorem~\ref{thm:surjectivity-Bartnik}]
We will show that for arbitrarily fixed $(\tau, \phi) \in \C^{2,\alpha}(\Sigma)\times \C^{1,\alpha}(\Sigma)$ and for a given $f\in \C^{0,\alpha}_{-q}(M)$, there is $h\in \C^{2,\alpha}_{-q}(M)$ solving 
\begin{align}\label{eq:Bartnik-boundary}
\begin{split}
	&L_g (h) = f\quad \mbox{ in } M\\
	&\left \{ \begin{array}{l} h^\intercal = \tau \\ DH|_g(h) = \phi \end{array} \right. \quad \mbox{ on } \Sigma.
\end{split}
\end{align}
The following argument is standard (see, e.g. Corollary 2.6 in \cite{Anderson:2019tm}). Let $\{ e_0=\nu, e_1,\dots, e_{n-1}\}$ be a local orthonormal frame near $\Sigma$. We can find a compactly supported $(0,2)$-tensor $\hat{\tau}$ such that  for $\alpha, \beta=1,\dots, n-1$, $\hat{\tau}_{\alpha \beta} = \tau_{\alpha \beta}$, $\hat{\tau}_{\alpha 0}=0, \hat{\tau}_{00}=0$, and  $\nu( \sum_\alpha \hat{\tau}_{\alpha \alpha} ) =2\phi$  on $\Sigma$.  By construction and using \eqref{eq:mean-curvature},
\[
	(\hat{\tau}^\intercal, DH|_g(\hat{\tau})) = (\tau, \phi). 
\]
For this $\hat{\tau}$, we can find $h_0\in \C^{2,\alpha}_{-q}(M)$ with $(h_0^\intercal, DH|_g(h_0) )=0$ on $\Sigma$ that satisfies the equation $L_g (h_0)=-L_g (\hat{\tau})+f$ in $M$  by Proposition~\ref{prop:surjectivity 2}. We then let 
\[
	h =h_0+ \hat{\tau}.
\]
\end{proof}

\section{Mass minimizers are static}\label{sec:static}

We will apply  Theorem~\ref{thm:boundary0} to characterize a mass minimizer  and prove Theorem~\ref{thm:static1} and Theorem~\ref{thm:positive}.  We let $(\M, \mb)$ be a static, reference metric with a static potential $V_0$. We let $(M, g)$ be an ALH manifold with respect to $(\M, \mb)$. Denote by $\mathcal{U}$ an open neighborhood of $g$ in $\C^{2,\alpha}_{-q}(M)$ that contains only Riemannian metrics. 

Let $V$ be a scalar function such that $V-V_0\in \C^2_{1-q}(M)$. Define the functional $\mathcal{F}_V$ for $\gamma\in \mathcal{U}$ by 
\begin{equation}\label{eq:functional F}
	\mathcal{F}_V(\gamma)=m(\gamma, V)-\int_M V\big(R_\gamma+n(n-1)\big)\,d\mu_g,
\end{equation}
where we recall the definition of the mass integral $m(\gamma, V)$ in \eqref{eq:mass2}. 

The above functional does not obviously take finite values since, without further assuming $r (R_\gamma + n(n-1) )\in L^1(M)$,  $m(\gamma, V)$ and the volume integral may individually diverge. The following lemma gives an alternative expression~\eqref{eq:alternative} for $\mathcal{F}_V$, from which one can see that $\mathcal{F}_V$ does take finite values (see the proof of Lemma \ref{lemma:massfinite}). 
\begin{lemma}[{Cf. \cite[Lemma 4.1]{Huang:2020tm}} and \cite{Anderson:2013vx}]\label{lemma:functional}
The functional $\mathcal{F}_V:\mathcal{U}\to \mathbb{R}$ can be alternatively expressed as 
\begin{align}\label{eq:alternative}
\begin{split}
	\mathcal{F}_V(\gamma) &= \int_M\Big(V\big( L_g e- R_\gamma - n(n-1)\big)  -e\cdot L_g^* V \Big)\, d\mu_g\\
	&\quad + \int_{\Sigma} \left[V\big(\mathrm{div}\, e (\nu) -\nu (\tr  \, e) \big)+ (\tr  \,e )\,  \nu (V) - e(\nabla V, \nu)\right] \, d\sigma_g
\end{split}
\end{align}
where $e= \gamma -\mb$ (we extend $\mb$ to be a smooth Riemannian metric everywhere on $M$), the unit normal~$\nu$ with respect to $g$ points to infinity, and the geometric operators  are all with respect to $g$. 

Furthermore, the first variation of $\mathcal{F}_V$ at $g$ is given by, for any symmetric $(0,2)$-tensor $h$,
\begin{align}\label{eq:first}
	D\mathcal{F}_V|_g(h)=-\int_M h \cdot L_g^* V\, d\mu_g + \int_\Sigma \Big( -V\big(2DH|_g(h)+A_g\cdot h^\intercal\big) + \nu(V) \,\tr  h^\intercal \Big)\,d\sigma.
\end{align}
\end{lemma}
\begin{proof}
By \eqref{eq:mass}, we have 
\begin{align*}
	m(\gamma, V) &= \int_{M} V L_g e\, d\mu_g - \int_{M} e\cdot L_g^* V \, d\mu_g\\
	&\quad  + \int_{\Sigma} \left[V\big(\mathrm{div}\, e (\nu) -\nu (\tr  \, e) \big)+ (\tr  \,e )\,  \nu (V) - e(\nabla V, \nu)\right] \, d\sigma_g.
\end{align*}
Substituting the above formula for $m(\gamma, V)$ into \eqref{eq:functional F} gives the desired expression~\eqref{eq:alternative}. 

From \eqref{eq:alternative}, we compute the first variation: 
\begin{align}
\label{eq:first2}
	D\mathcal{F}_V|_g(h)=-\int_M h \cdot L_g^* V\, d\mu_g +  \int_{\Sigma} \left[V\big(\mathrm{div}\, h (\nu) -\nu (\tr  \, h) \big)+ (\tr  \,h )\,  \nu (V) - h(\nabla V, \nu)\right] \, d\sigma_g.
\end{align}
One can re-express the boundary integrands as follows (see e.g. \cite[Proposition 3.1]{Anderson:2013vx}): 
\begin{align*}
	V(\mathrm{div} h(\nu) - \nu(\tr \, h ) )&= V(-2DH|_g(h) - A_g\cdot h^\intercal -  \mathrm{div}_\Sigma \omega)\\
	(\tr \, h) \, \nu (V) - h(\nabla V, \nu) &= (\tr \, h^\intercal ) \nu(V) - \omega ((\nabla V)^\intercal)
\end{align*}
where $\omega$ is the one-form on the tangent bundle of $\Sigma$ defined by $\omega(\cdot) = h(\nu, \cdot)$.  Substituting the right hand sides above into  \eqref{eq:first2} and applying integration by parts to eliminate the term involving $\omega$ gives the stated first variation formula.

\end{proof}

We prove Theorem~\ref{thm:static1}. We restate the statement and also spell out the assumption~($\star$); the neighborhood $\mathcal{U}_\star$ in the statement is defined in \eqref{eq:neighborhood}.

\begin{manualtheorem}{\ref{thm:static1}}
Let $(\M, \mb)$ be a static, reference metric with a static potential $V_0$. Let $(M,g)$ be an ALH manifold  with respect to $(\M, \mb)$, having nonempty boundary $\Sigma$. Suppose
\begin{itemize}  
\item[($\star$)] There is a neighborhood $\mathcal{U}_{\star}$ of $g$ such that for any $\gamma\in \mathcal{U}_{\star}$ with $R_\gamma = R_g$ in $M$, we have $m(\gamma, V_0)\ge m(g, V_0)$. 
\end{itemize}
 	Then $(M, g)$ is static with a static potential $V$ satisfying $V -V_0=O(r^{-d})$ for some number $d>0$.
\end{manualtheorem}




\begin{proof}
We extend $V_0$ to be defined everywhere in $(M, g)$, and define  the functional $\mathcal{F}_{V_0}$ as in \eqref{eq:functional F}. The subscript $V_0$ will be omitted for the rest of proof.

Let $\mathcal U$ be a small open neighborhood of $g$ in $g+\C^{2, \alpha}(M)$. The assumption~($\star$)  implies that the functional $\mathcal{F}: \mathcal U\to \mathbb R$ achieves a local minimum at $g$ subject to the constraints $R_\gamma = R_g$ in $M$ and $(\gamma^\intercal, H_\gamma) = (g^\intercal, H_g)$ on $\Sigma$. Theorem~\ref{thm:boundary0}  implies  that the linearization of the constraints
\[
	T:   \C^{2,\alpha}_{-q}(M)\longrightarrow \C^{0,\alpha}_{-q}(M) \times \C^{2,\alpha}(\Sigma)\times \C^{1,\alpha}(\Sigma)
\] 
is surjective, where recall that $T(h) = (L_gh, h^\intercal, DH|_g(h))$ as defined in \eqref{eq:T}. 

We can then apply the Method of Lagrange Multiplier (see, e.g. \cite[Theorem D.1]{Huang:2020um}) and find a ``Lagrange multiplier'' $(\lambda, \mu, \nu) \in \big( \C^{0,\alpha}_{-q}(M) \big)^* \times \big(\C^{2,\alpha}(\Sigma)\big)^* \times \big(\C^{1,\alpha}(\Sigma)\big)^*$ so that 
\[
D\mathcal{F}|_g(h) = \lambda (L_gh) + \mu(h^\intercal) + \nu(DH|_g(h)) \quad \mbox{ for any }  h\in  \C_{-q}^{2,\alpha}(M).
\]
Taking $h\in \C^\infty_c(\Int M)$ yields
\[
D\mathcal{F}|_g(h) = \lambda (L_gh) \quad \mbox{ for any }  h\in  \C^\infty_c(M).
\]
Together with the first variation formula of $\mathcal{F}$ in \eqref{eq:first}, 
it follows that $\lambda$ is a weak solution (as a distribution) to the equation
	\[
		L_g^*\lambda =-L_g^* V_0\in  \C^{0,\alpha}_{1-q}(M).
	\]
	By elliptic regularity, $\lambda\in \C^{2,\alpha}_{\text{loc}}(M)$, see e.g. \cite[Lemma 6.33]{Folland:1999}.  By Lemma~\ref{lemma:linear}, we have either $\lambda$ grows linearly in a cone or $\lambda=O(r^{-d})$ for some $d>0$. But  $\lambda$ cannot grow linearly in a cone $U$ because  $\lambda\in \big(\C^{0,\alpha}(M)\big)^*$ and by Remark~\ref{remark:dual}.
	
 To conclude, we show that $V:=V_0+\lambda$ is a static potential of $(M,g)$.

\end{proof}

We next prove Theorem~\ref{thm:positive}. The previous theorem already shows a mass minimizer must be static, but in order to characterize a static metric, we would like to obtain a \emph{positive} static potential with additional boundary properties. We recall the statement and also spell out the assumption $(\star\star)_{H_0}$.

\begin{manualtheorem}{\ref{thm:positive}}
Let $(\M, \mb)$ be a static, reference metric with a static potential $V_0>0$ near infinity.	Let $(M,g)$ be an ALH manifold  with respect to $(\M, \mb)$, having nonempty boundary $\Sigma$, and let $H_0$ be a function (can be constant) such that $H_g\le H_0$.  Suppose that 
\begin{itemize}  
\item[$(\star\star)_{H_0}$] There is an open subset $\mathcal{U}$ of $g$ in $g+\C^{2,\alpha}_{-q}(M)$ such that for any $\gamma\in \mathcal{U}$ with $R_\gamma = R_g$ in $M$ and $H_\gamma \le H_0$ on $\Sigma$,  we have $m(\gamma, V_0)\ge m(g, V_0)$.
\end{itemize}
Then the following holds:
\begin{enumerate}
\item $(M, g)$ is static with a static potential $V$ satisfying $V -V_0 =O(r^{-d})$ for some number $d>0$.\label{item:static}
\item  $VA_g= \nu(V)g^\intercal$ on $\Sigma$.  \label{item:umbilic}
\item The static potential $V>0$ everywhere in $M$.\label{item:positive}
\item  $\Sigma$ has mean curvature $H_g =H_0$.   \label{item:constant}
\end{enumerate}
\end{manualtheorem}
\begin{remark}\label{remark:AJ}
Our proof of $V>0$ is inspired by the variational argument in \cite[Theorem 2.10]{Anderson:2019tm} for asymptotically flat manifolds,  some of which also originated from the rigidity proof of \cite{Schoen:1979wh}. The proof in \cite{Anderson:2019tm} uses that the scalar curvature map is a submersion, which seems unknown yet in our setting. So we use an alternative argument by conformal transformation, stated as \emph{Claim} inside the following proof. 
\end{remark}
\begin{remark}\label{remark:positive}
For applications in the next section, we will mainly be concerned with $H_0 = n-1$. Then the above conclusions imply that $V= \nu(V)$  and $A_g = g^\intercal$ on $\Sigma$. 
\end{remark}
\begin{proof}[Proof of Theorem \ref{thm:positive}]
Since $(\star\star)_{H_0}$ implies $(\star)$ in Theorem~\ref{thm:static1}, Item~\eqref{item:static} follows. In particular, we have $R_g = -n(n-1)$. We first prove a general fact  that we can vary $g$ to get a family of metrics of the same constant scalar curvature and prescribe the variation of the Bartnik boundary data. \\

\noindent {\bf Claim:} Given an arbitrary pair of a symmetric $(0,2)$-tensor $\tau\in \C^{2,\alpha}(\Sigma)$ and a scalar function~$\phi\in \C^{1,\alpha}(\Sigma)$ on~$\Sigma$,  there is a differentiable family of metrics $g(t)$, for $|t|$ small, such that 
\begin{align}\label{eq:Bartnik}
\begin{split}
	g(0)&=g, \quad R_{g(t)} = R_g =  -n(n-1)\quad  \mbox{ in } M\\
	w^\intercal &= \tau,\quad \mbox{ and } \quad DH|_g(w) = \phi \quad \mbox{ on } \Sigma
\end{split}
\end{align}
where $w=g'(0)$. As a consequence of Lemma~\ref{lemma:functional}, the first variation of the mass integral for $g(t)$ is given by  
\begin{align}\label{eq:mass-decreasing}
	\left. \frac{d}{d t}\right|_{t=0} m(g(t), V_0) &=\int_\Sigma \Big( -2V\phi + \tau \cdot \big(-VA_g+ \nu(V) g^\intercal\big) \Big)\,d\sigma_g.
\end{align}

\noindent\emph{Proof of Claim.} By Theorem~\ref{thm:boundary0},  there is a symmetric $(0,2)$-tensor  $h\in \C^{2,\alpha}_{-q}(M)$ such that 
\[
(L_g h, h^\intercal, DH|_g(h)) = (0,\tau, \phi).
\] 
Define $\hat{g}(t) = g+th$. For each $t$, we would like to solve for a scalar function $u$ with $u - 1\in \C^{2,\alpha}_{-q}(M)$  and $\nu_{\hat{g}(t)}(u)=0$ on $\Sigma$ such that $ g(t) := u^{\frac{4}{n-2}} \hat{g}(t)$ has constant scalar curvature $-n(n-1)$. By conformal transformation formula, it is equivalent to solving the following system for $u\in 1+\C^{2,\alpha}_{-q}(M)$
\begin{align}\label{eq:conformal}
\begin{split}
	\Delta_{\hat{g}(t)} u- \tfrac{n-2}{4(n-1)} R_{\hat{g}(t)} u &= \tfrac{n(n-2)}{4} u^{\frac{n+2}{n-2}} \quad \mbox{ in } M\\
	\nu_{\hat{g}(t)}(u)&=0 \quad \mbox{ on } \Sigma.
\end{split}
\end{align}
It is clear that at $t=0$, $u\equiv 1$ is a solution.  We show how to obtain $u(t)$ for $|t|$ small by the inverse function theorem: Fix $t$ and rewrite the above system as the map $T: 1+\C^{2,\alpha}_{-q}(M)\to \C^{0,\alpha}_{-q}(M)\times \C^{1,\alpha}(\Sigma)$
\[
T(u) = \left(\Delta_{\hat{g}(t)} u- \tfrac{n-2}{4(n-1)} R_{\hat{g}(t)} u - \tfrac{n(n-2)}{4} u^{\frac{n+2}{n-2}},\, \nu_{\hat{g}(t)}(u)\right).
\]
Denote the linearization of $T$ at $u=1$ by $DT|_1: \C^{2,\alpha}_{-q}(M)\to \C^{0,\alpha}_{-q}(M)\times \C^{1,\alpha}(\Sigma)$ and compute
\begin{align*}
DT|_1(v) &= \left(\Delta_{\hat{g}(t)} v - \Big( \tfrac{n-2}{4(n-1)} R_{\hat{g}(t)} + \tfrac{n(n+2)}{4}\Big)v, \, \nu_{\hat{g}(t)}(v)\right).
\end{align*}
 Since  $DT|_1$ converges to $(\Delta_g v -n v, \nu_g (v))$ as $t\to 0$ and the latter map is an isomorphism by Lemma~\ref{lemma:isomorphism} and Remark~\ref{rmk:Neumann}, we conclude that for $|t|$ sufficiently small, each $DT|_1$ is an isomorphism and thus the corresponding map $T$ is a local diffeomorphism at $1$ by the inverse function theorem. Therefore, there is $u(t)$ solving~\eqref{eq:conformal}. The family of solutions $u(t)$ is differentiable in $t$ by  smooth dependence, so is $g(t) = u(t)^{\frac{4}{n-2}} \hat{g}(t)$. 
 
 We have shown $g(t)$ has constant scalar curvature $-n(n-1)$. We now compute variations of the Bartnik boundary data in \eqref{eq:Bartnik}. We begin by computing $u'(0)$.  Differentiating \eqref{eq:conformal} in $t$ at $t=0$ and using $u(0)=1$, $\left.\frac{\partial}{\partial t} \right|_{t=0}R_{\hat{g}(t)} = L_gh=0$, we get 
\begin{align*}
	\Delta_{g} u'(0)- n u'(0)&= 0\quad \mbox{ in } M\\
	\nu_{g} (u'(0)) &= 0\quad \mbox{ on } \Sigma.
\end{align*}
Thus $u'(0)$ is identically zero (see Lemma~\ref{lemma:isomorphism}). It is then direct to see that 
\[
	w:=g'(0) = \hat{g}'(0) = h.
\]
In particular, $w^\intercal= h^\intercal = \tau$ on $\Sigma$. To compute the mean curvature, we apply conformal transformation formula and use $\nu_{\hat{g}(t)}(u(t))=0$:
\[
	H_{g(t)} = u(t)^{\frac{-2}{n-2}} \left( H_{\hat{g}(t)} + \tfrac{2(n-1)}{n-2} u(t)^{-1} \nu_{\hat{g}(t)} (u(t)) \right) = u(t)^{\frac{-2}{n-2}}  H_{\hat{g}(t)}.
\] 
Differentiating the mean curvature identity in $t$ gives $DH|_g(w) = DH|_g(h) = \phi $. 

Let $\mathcal{F}_V$ be defined as \eqref{eq:functional F}. Since $g(t)$ has constant scalar curvature $-n(n-1)$ and $V$ is asymptotic to $V_0$, we have 
\[
	\mathcal{F}_V(g(t)) = m(g(t), V) = m(g(t), V_0).
\] 
Then the first variation formula \eqref{eq:first} says
\begin{align*}
	 \left. \frac{d}{d t}\right|_{t=0} m(g(t), V_0) &= \left. \frac{d}{d t}\right|_{t=0} \mathcal{F}_V(g(t)) = D\mathcal{F}_V|_g(w)\\
	&=\int_\Sigma \Big( -V\big(2DH|_g(w)+A_g\cdot w^\intercal\big) + \nu(V) \,\tr \, w^\intercal \Big)\,d\sigma_g\\
	&= \int_\Sigma \Big( -V\big(2\phi +A_g\cdot \tau\big) + \nu(V) \,\tr \, \tau \Big)\,d\sigma_g,
\end{align*}
where we use that $L_g^*V=0$ in the second line and $(w^\intercal, DH|_g(w) )= (\tau, \phi)$ in the third line.
Rearranging the integrands gives \eqref{eq:mass-decreasing}.

 \qed

The rest of the proof will proceed as follows: If, to get a contradiction, the desired conclusions do not hold, we can choose suitable $\tau$ and $\phi$ in \eqref{eq:Bartnik} to make \eqref{eq:mass-decreasing} strictly negative and also to ensure $H_{g(t)}< H_0$ for $t>0$. Then we get 
 contradiction to the assumption~$(\star\star)_{H_0}$.

 \vspace{10pt}
\noindent{\bf Item \eqref{item:umbilic}:} We prove $VA_g=\nu(V) g^\intercal $ on $\Sigma$. Suppose on contrary that $-VA_g+\nu(V) g^\intercal$ is not identically zero. We can find $\tau$ on $\Sigma$ such that 
\[
\int_\Sigma  \tau\cdot \big(-VA_g+\nu(V)g^\intercal \big)\, d\sigma_g<0.
\]
Then let $\phi<0$ on $\Sigma$ with $|\phi|$ sufficiently small such that \eqref{eq:mass-decreasing} is also negative: 
\[
	\left. \frac{d}{d t}\right|_{t=0} m(g(t), V_0)=\int_\Sigma \Big(-2V\phi + \tau\cdot \big(-VA_g+\nu(V)g^\intercal \big) \Big)\, d\sigma_g<0.
\]
Let $g(t)$ be a family of metrics as in \eqref{eq:Bartnik} for the above choice of $(\tau, \phi)$.  Therefore, for $t>0$, $H_{g(t)}< H_g\le H_0$ (because $\phi<0$) and $m(g(t), V_0) < m(g, V_0)$. It contradicts $(\star\star)_{H_0}$.

 \vspace{10pt}
\noindent{\bf Item \eqref{item:positive}:} We prove $V>0$ in $M$.  Note that $V>0$ near infinity (because $V_0>0$ near infinity by assumption). We first show that $V\ge 0$ on $\Sigma$, and thus   $V>0 $ in $\Int M$ by applying strong maximum principle to $\Delta_g V - nV=0$. Suppose, to get contradiction, that $V<0$ somewhere in $\Sigma$. We can find a function $\phi < 0$ on $\Sigma$ and 
\[
	\int_\Sigma -2  V\phi \, d\sigma_g <0.
\] 
We let $g(t)$ be the family of metrics from \eqref{eq:Bartnik} with $\tau = 0$ and $\phi$ as chosen. A similar argument as above gives contradiction to $(\star\star)_{H_0}$. To show that $V>0$ on $\Sigma$, we note that if $V(p)=0$ for some $p\in \Sigma$, then $\nu(V)(p)=0$ by Item~\eqref{item:umbilic}, and the static equation implies that $V(\gamma(t))=0$ is along the normal geodesic $\gamma(t)$ in $\Int M$  emanating from $p$, which violates that $V>0$ in $\Int M$.

\vspace{10pt}

 \noindent{\bf Item \eqref{item:constant}:} We prove that $\Sigma$ has mean curvature $H_g =H_0$. Suppose on the contrary $H_g<H_0$  somewhere. Choose a function $\phi$ on $\Sigma$ such that  $\phi <0$ on the set where $H_g = H_0$ and the zero set of $\phi$ is contained in the subset where $H_g <H_0$ (in other words,  $\{ p\in \Sigma: H_g(p) = H_0(p) \}\subsetneq \mathrm{supp} (\phi^-)$), and 
  \begin{align*}
 &\int_\Sigma -2V\phi \, d\sigma_g < 0 \quad \mbox{(note that $V >0$ from Item~\eqref{item:positive})},
 \end{align*}
Note that $\phi$ is necessarily positive somewhere in the subset where $H_g<H_0$, in order for the above integral condition to hold. Let $g(t)$ be a family of metrics as in \eqref{eq:Bartnik} for $\tau=0$ and the above choice of $ \phi$. The conditions on $\phi$ ensure that $H_{g(t)} < H_0$ for $t>0$ small.  We again get a contradiction.

\end{proof}

A main conclusion of Theorem~\ref{thm:positive} is to show positivity of a static potential. We note in the next lemma that under the given special boundary condition, one may be able to obtain positivity of $V$ in $\Int M$. The proof is a direct generalization of an argument for asymptotically flat manifolds (see, e.g. \cite[Section 1]{Huang:2018ue}). Note that this result is not used elsewhere in the paper.

\begin{lemma}\label{lem:nonnegative static potential}
	Let $(M,g)$ be a  ALH manifold with boundary $\Sigma$. Suppose $(M, g)$ is static with a static potential $V$ satisfying $V>0$ near infinity. Suppose $\Sigma$ is a locally outermost,  locally area-minimizing, minimal hypersurface. Then  $V>0$   in $ \Int M$ and $V=0$ on $\Sigma$.
\end{lemma}
\begin{proof}
We describe how the  arguments in  \cite[Section 1]{Huang:2018ue} for asymptotically flat manifolds apply in our setting. The static equation for $V$ and that $\Sigma$ has zero mean curvature imply  
\[
	0 = \Delta_g V - nV = \Delta_\Sigma V + \nabla^2 V(\nu, \nu) - nV= \Delta_\Sigma V + \Ric(\nu, \nu) V. 
\]	
Since $\Sigma$ is in particular a stable minimal hypersurface, the stability implies that the least eigenvalue of $-\Delta_\Sigma u - (|A_g|^2+ \Ric(\nu, \nu))u$ is non-negative, and thus either $V$ is its first eigenfunction (thus $V$ has no zeros on $\Sigma$ and the second fundamental form $A_g=0$ on $\Sigma$) or $V$ is identically zero on $\Sigma$.  We will rule out the former case:  we may without loss of generality assume $V>0$ on $\Sigma$, as the case $V<0$ on $\Sigma$ can be argued similarly. By Galloway's monotonicity formula \cite[Lemma 3]{Galloway:1993tu}, consider the family hypersurfaces $\{\Sigma_t\}$ with $\Sigma_0=\Sigma$ and $\frac{\partial}{\partial t} \Sigma_t = V \nu$ where $\nu$ points to infinity. Then $\frac{\partial}{\partial t} \left( \frac{H_{\Sigma_t}}{V}\right) = -|A_g|^2$  on $\Sigma_t$ for all $t\ge 0$ small. The locally area-minimizing property implies that $\Sigma_t$, for $t>0$, form a foliation of totally geodesic hypersurfaces, but that contradicts the locally outermost property. 
To conclude, we show that  $V\equiv 0$ on $\Sigma$, and thus $V>0$ in $\Int M$ by the strong maximum principle. 

\end{proof}

\section{Static uniqueness and rigidity of the positive mass theorems}\label{sec:uniqueness}

The goal in this section is to characterize a mass minimizer  and prove Theorem~\ref{thm:uniqueness}, which then implies Theorem~\ref{thm:rigidity1} and Theorem~\ref{thm:rigidity0}. By Theorem~\ref{thm:positive}, it suffices to establish static uniqueness, Corollary~\ref{corollary:static} below. There have been various static uniqueness results by, for example, \cite{Wang:2005aa,Galloway:2015tj,Chrusciel:2020wg}. The boundary conditions that naturally arise in our setting seem to be different from theirs, but the proof is also based on the following fundamental identity of Y.~Shen~\cite{Shen:1997ug} and X. Wang~\cite{Wang:2005aa} (see also Lemma 4.4 and Proof of Theorem 2 in \cite{Huang:2020tm}): Let $(M, g)$ be a static,  ALH manifold with boundary $\Sigma$. 
Suppose the static potential $V >0 $ in $\Int M$. Then  
\begin{equation}\label{eq:mass formula}
	\int_{M} V^{-1} |\nabla^2 V-Vg|^2\, d\mu_g	=-\tfrac{n-2}{2}m(g, V)-\int_{\Sigma} \big(\Ric  + (n-1) g \big) (\nabla V, \nu) \, d\sigma_g
\end{equation}
where the unit normal on $\Sigma$ points to infinity. 

We shall see that boundary conditions obtained from  Theorem~\ref{thm:positive} and Remark~\ref{remark:positive} imply the boundary integral vanishes. 
\begin{lemma}\label{lemma:boundary}
Let $(M, g)$ be a static, ALH manifold with a static potential $V>0$. Suppose that $\Sigma$ is umbilic with mean curvature $(n-1)$ and that $V = \nu(V)$ on $\Sigma$. Then 
\[
	\int_{\Sigma} \big(\Ric  + (n-1) g \big) (\nabla V, \nu) \, d\sigma_g=0.
\]
As a consequence, if $m(g, V)=0$, then $V$ satisfies $\nabla^2 V-Vg=0$ in $M$.
\end{lemma}
\begin{proof}
Let $\{ \nu, e_1,\dots, e_{n-1}\}$ be a local orthonormal frame near $\Sigma$ and $\alpha, \beta\in \{ 1,\dots, n-1\}$. We will show that, on $\Sigma$, 
\begin{align}\label{eq:boundary-integral}
\begin{split}
	\Ric (e_\alpha, \nu)&=0\\
	\int_\Sigma V \big(\Ric(\nu, \nu) + (n-1)\big)\, d\sigma_g &= 0.
\end{split}
\end{align}
Using that $V = \nu(V)$ and $A_g = g^\intercal$ on $\Sigma$, we compute
\begin{align*}
	0&=e_\alpha  (V-\nu(V)) = e_\alpha(V) - e_\alpha (\nu (V))\\
	&= e_\alpha(V) - \nabla^2 V(e_\alpha, \nu) - \nabla_{\nabla_{e_\alpha} \nu } V\\
	&= e_\alpha(V)  - V\Ric(e_\alpha, \nu)  - A_{\alpha \beta} e_\beta (V)\\
	&=- V\Ric(e_\alpha, \nu),
\end{align*}
where we use the static equation $\nabla^2 V = V (\Ric + ng)$. Since $V>0$, we get $\Ric(e_\alpha, \nu)=0$. Next, we use again $\nabla^2 V = V (\Ric + ng)$, $V= \nu(V)$, and $H_g = n-1$ to compute
\begin{align*}
	0 &= \Delta_g V - nV \\
	&= \Delta_\Sigma V + \nabla^2 V(\nu, \nu) + H_g \nu(V) - nV\\
	&= \Delta_\Sigma V + V \big( \Ric(\nu, \nu) + (n-1) \big).
\end{align*}
Integrating the previous identity gives the second identity in \eqref{eq:boundary-integral}. We now use \eqref{eq:boundary-integral} to compute the boundary integral:
\begin{align*}
	\int_{\Sigma} \big(\Ric  + (n-1) g \big) (\nabla V, \nu) \, d\sigma_g &= \int_{\Sigma} \nu(V) \big(\Ric  + (n-1) g \big) (\nu, \nu) \, d\sigma_g\\
	&= \int_{\Sigma} V\big(\Ric(\nu, \nu)  + (n-1) \big)\, d\sigma_g=0.
\end{align*}

To complete the proof, we use  \eqref{eq:mass formula} to get that the volume integral over $M$ is identically zero and hence $\nabla^2 V - Vg=0$. 
\end{proof}

A manifold that admits a positive function satisfying $\nabla^2 V-Vg=0$ can be characterized as in \cite{Tashiro:1965uj,Kanai:1983ut} for complete manifolds without boundary and in \cite[Proposition 4.1]{Galloway:2020to} for manifolds with boundary. We show how to extend their argument for the boundary conditions in our applications. 

\begin{lemma}[Cf. {\cite{Tashiro:1965uj,Kanai:1983ut,Galloway:2020to}}]\label{lemma:Obata}
Let $(\M, \mb)$ be a reference manifold with the conformal infinity $(N, h)$ of type $k$. Let $(M, g)$ be ALH with respect to $(\M, \mb)$  with possibly nonempty boundary $\Sigma$. Suppose there is a function $V$ satisfying $V>0$ in $\Int M$ and $\nabla^2 V = Vg$ in $M$
\begin{enumerate}
\item \label{item:complete} If $M$ is boundaryless, then $k=1$ and $(M, g)$ is isometric to the standard hyperbolic space. 
\item  \label{item:split}Suppose the boundary $\Sigma$ is nonempty, and suppose $V$ is constant on $\Sigma$. Then $(M, g)$ is isometric to the manifold $[c, \infty)\times N$ for some $c\in \mathbb{R}$ with the warped product metric of the form
\[
	dt^2 + \xi(t)^2 h
\]
where 
\[
	\xi (t) = \left\{ \begin{array}{ll} \sinh t & \mbox{ for } k=1 \\ e^t & \mbox{ for } k=0 \\ \cosh t & \mbox{ for } k=-1\end{array}\right..
\]
\item \label{item:CMC}Suppose the boundary $\Sigma$ is nonempty with mean curvature $H_\Sigma=(n-1)$, and suppose $\nu(V)>0$ on $\Sigma$. Then $k=0$, $V$ is constant on~$\Sigma$, and Item~\eqref{item:split} holds.  
\end{enumerate}
\end{lemma}
\begin{proof}
Item~\eqref{item:complete} is from \cite{Tashiro:1965uj,Kanai:1983ut}. We just note how to exclude the cases $k=0, -1$.  A complete manifold without boundary in the case $k=0$ has a cuspidal end,  and in the case $k=-1$ has \emph{two} ALH ends, both are excluded by our assumption that $(M, g)$ has one ALH end.

Item~\eqref{item:split} is obtained in \cite[Proposition 4.1]{Galloway:2020to}. We discuss how Item~\eqref{item:CMC} can be obtained from their argument and  Item~\eqref{item:split}.  It is  shown that $V\to \infty$ as $t\to \infty$ and $V$ has no interior critical point in~\cite[Proposition 4.1]{Galloway:2020to}. Let $a= \max_\Sigma V$. Since $V$ has no interior critical point, together with the assumption that $\nu(V)>0$, we see that $V^{-1}(a)$ is a regular hypersurface. Let $\Sigma_a$ be the part of $V^{-1}(a)$ that is homologous to $\Sigma$, and let $\Omega_a$ be the open set bounded between $\Sigma_a$ and $\Sigma$.  (Here we use that $V\to \infty$ so $\Sigma$ must be separated from the infinity by $V^{-1}(a)$.) We can apply Item~\eqref{item:split} to $(M\setminus \Omega_a, g)$. In particular, its boundary $\Sigma_a$ must have constant mean curvature~$H_{\Sigma_a}$ (see Example~\ref{ex:CMC}), which according to the type~$k$,  
\[
	H_{\Sigma_a}  \left\{  \begin{array}{ll} > n- 1 & \mbox{ for } k=1\\ 
	= n-1 &\mbox{ for } k=0\\
	< n-1 &\mbox{ for } k=-1\end{array} \right..
\]
Because $\Sigma_a$ is tangent to $\Sigma$ from outside, by maximum principle, the mean curvature 
\[
H_{\Sigma_a}\le  H_\Sigma=n-1
\] with equality if and only if $\Sigma_a$ is identical to $\Sigma$. That excludes the case $k=1$. For the case $k=0$, we see that ${\Sigma_a} = \Sigma$, so the conclusion follows. For the case $k=-1$, we can apply the flow argument~\cite[Proposition 4.1]{Galloway:2020to} on $\Sigma_a$ into $\Omega_a$ and the same argument show that $\Omega_a$ is isometric to a subset of the reference manifold $(\M, \mb)$ of type $k=-1$, but that reference manifold cannot contain a closed hypersurface $\Sigma$ with mean curvature $n-1$ because it is foliated by hypersurfaces of constant mean curvature $<n-1$. We then exclude the case $k=-1$. 
\end{proof} 

\begin{corollary}[Static uniqueness for zero mass] \label{corollary:static}
Let $(\M, \mb)$ be a reference manifold with the conformal infinity $(N, h)$ of type $k$. Let $(M, g)$ be ALH with respect to $(\M, \mb)$. Suppose $(M, g)$ is static with a static potential $V>0$ in $M$ and with $m(g, V)=0$. In the case that $M$ has nonempty boundary $\Sigma$, we assume\footnote{See Remark~\ref{remark:positive} for the motivation of the boundary assumption.} that $\Sigma$ is umbilic with mean curvature $(n-1)$ and that $V = \nu(V)$ on $\Sigma$. Then we must have $k=1, 0$, and $(M, g)$ is characterized as follows:
\begin{enumerate}
\item $k=1$: $(M, g)$ is the standard hyperbolic space without  boundary.  
\item  $k=0$: $(M, g)$ is isometric to a Birmingham-Kottler manifold  $\left([1,\infty)\times N, r^{-2} dr^2 + r^2 h\right)$ whose conformal infinity $(N, h)$ is Ricci flat. (In particular, $g$ itself is Poincar\'e-Einstein.)
\end{enumerate}
\end{corollary}
\begin{proof}
By \eqref{eq:mass formula} (if $M$ is boundaryless) or Lemma~\ref{lemma:boundary} (if $M$ has boundary), the positive static potential satisfies $\nabla^2 V- Vg=0$, which also implies that $g$ is Poincar\'e-Einstein by the static equation. The conclusion follows from Lemma~\ref{lemma:Obata}. 
\end{proof}

We now prove Theorem~\ref{thm:uniqueness}. Theorem~\ref{thm:rigidity1} and Theorem~\ref{thm:rigidity0} are direct consequences as discussed in Introduction. 

\begin{proof}[Proof of Theorem~\ref{thm:uniqueness}]
By Theorem~\ref{thm:positive}, $(M, g)$ is static with a static potential $V>0$ in $M$, and in the case that $M$ has nonempty boundary $\Sigma$, we also have that $\Sigma$ is umbilic with mean curvature $n-1$ and $V=\nu(V)$ on $\Sigma$. 
The conclusion then follows from  Corollary~\ref{corollary:static}.
\end{proof}

\subsection{Miscellaneous results}

We include results of independent interest. They are not used elsewhere in the paper. 

\subsubsection{Bartnik mass minimizers}\label{sec:Bartnik}

Given a  static, reference manifold $(\M, \mb)$ with a static potential $V_0$ and a compact manifold $(\Omega, g_0)$ with boundary $\Sigma$, we define the (hyperbolic) Bartnik mass with respect to $(\M, \mb)$ to be
\[
	m_B(\Omega, g_0) = \inf \big\{ m(g, V_0):  (M, g) \mbox{ is admissible}\big \},
\]
where  $(M, g)$ is said to be \emph{admissible} if the following holds
\begin{itemize}
\item $(M, g)$ is ALH with respect to $(\M, \mb)$ and has scalar curvature $R_g\ge -n(n-1)$. 
\item The boundary $\partial M$ is diffeomorphic to $\Sigma$, and the Bartnik boundary data match along $\partial M\cong \Sigma$:
\[
	g^\intercal=g_0^\intercal, \quad  H_g = H_{g_0}. 
\]
\item $(M, g)$ satisfies a non-degeneracy condition $\mathscr{N}$ that is ``open'' among small deformations of~$g$. As an example,  the condition $\mathscr{N}$ can say that $\partial M$ is strictly outward-minimizing in $(M, g)$, in the sense that it has volume strictly less than any hypersurface enclosing it. 
\end{itemize}
We also refer the definition to \cite[Section 3]{Martin:2018vp} and \cite{Cabrera-Pacheco:2018wc}. We say that an admissible $(M, g)$ is a \emph{Bartnik mass minimizer} of $(\Omega, g_0)$ if $m(g, V_0) = m_B(\Omega, g_0)$. 
As a direct consequence of Theorem~\ref{thm:static1}, we see that a Bartnik mass minimizer must be static. A different proof with a slightly different conclusion has been given by   {\cite[Theorem 3.0.1]{Martin:2018vp}}.
\begin{corollary}[Cf. {\cite[Theorem 3.0.1]{Martin:2018vp}}]
Let $(\M, \mb)$ and $(\Omega, g_0)$ be as above. If $(M, g)$ is a \emph{Bartnik mass minimizer} of $(\Omega, g_0)$, then $(M, g)$ is static with a static potential $V$ satisfying $V -V_0=O(r^{-d})$ for some number $d>0$.
\end{corollary}

\subsubsection{Static uniqueness}
We have seen in Example~\ref{ex:CMC} that the reference manifolds are foliated by hypersurfaces of constant mean curvature. We shall see that those hypersurfaces have further global properties:
\begin{itemize}
\item $k = 0$:  Since $(\M, \mb)$ is foliated by hypersurfaces with constant mean curvature $(n-1)$, the ``exterior region'' $[r, \infty)$ relative to $S_r$  cannot contain any closed hypersurface  with mean curvature strictly less than $(n-1)$. 
\item  $k = -1$: At $r=1$,  $S_{1}$ is a minimal hypersurface, and the ``exterior'' region $\M= (1,\infty)\times N$ is foliated by hypersurfaces of positive mean curvature and thus cannot contain any other minimal hypersurface. 
\end{itemize}
Those global properties are captured by the following ``outermost'' condition. 
\begin{definition}
Let $(M, g)$ be ALH with boundary $\Sigma$ having mean curvature $H_0$.	We say that $\Sigma$ is \emph{locally weakly outermost} if there is a collar neighborhood $U$ of $\Sigma$ such that there are no hypersurfaces in $U$ homologous to $\Sigma$ with mean curvature $<H_0$.
\end{definition}

A similar result by a similar proof  is obtained Chru\'sciel, Galloway, and Potaux in \cite[Theorem VI.7]{Chrusciel:2020wg} with the difference is that the ``locally weakly outermost'' condition here is replaced by their assumption that  $\Sigma$ is a (strongly) stable CMC hypersurface with mean curvature $n-1$ and that $V$ is constant on $\Sigma$. The point of this proposition is to find the boundary conditions entirely on $\Sigma$ but not on the static potential $V$.

\begin{proposition}
Let $(M,g)$ be ALH with boundary $\Sigma$. Suppose that $(M, g)$ is static with a static potential $V$ satisfying $V>0 $ in $\Int M$ and has the mass integral $m(g, V)=0$. Suppose the following boundary conditions hold:
\begin{align} \label{eq:BC}
\begin{split}
	&\mbox{$\Sigma$ has the mean curvature $H_g \le n-1$,  is locally weakly outermost, and}\\
	&\mbox{does not admit a metric of positive scalar curvature}.
\end{split}
\end{align}
Then the induced metric $h:=g^\intercal$ on $\Sigma$ is Ricci flat, and $(M,g)$ is isometric to $([1,\infty)\times \Sigma, dt^2 + e^{2t}h )$.
\end{proposition}

\begin{proof}


It is shown by Galloway and the second author in \cite[Theorem 1.3]{Galloway:2020to} that \eqref{eq:BC} implies that there is a collar neighborhood $U$ of $\Sigma$ such that $(U ,g|_U)$ is isometric to the warped product $\left([0,\varepsilon)\times\Sigma,dt^2+e^{2t}h\right)$ and $(\Sigma, h)$ is Ricci flat. From that, we can compute $(U, g|_U)$ has constant Ricci curvature $\Ric = (n-1) g$. Therefore, the boundary integral in \eqref{eq:mass formula} vanishes
\[
\int_{\Sigma} \big(\Ric  + (n-1) g \big) (\nabla V, \nu) \, d\sigma_g=\int_{\Sigma}  \nu(V) \big(\Ric  + (n-1) g \big) (\nu, \nu)\, d\sigma_g=0.
\]
Together with the assumption $m(g, V)=0$,  we see that $V$ satisfies $\nabla^2 V= Vg$ in $M$. The local splitting of $(U, g|_U)$ and that $\nabla^2 V= Vg$ in $U$ also imply $V=$ constant on $\Sigma$, proven in  \cite[Theorem 1.5]{Galloway:2020to}. The proof follows by Lemma~\ref{lemma:Obata}.
\end{proof}

\appendix

\section{An alternative proof to Anderson-Jauregui's scalar curvature surjectivity}\label{sec:af}

In this section we discuss Theorem~\ref{thm:AF}. All the notations in this section are self-contained, and should not be confused with the rest of the paper. (The weighted H\"older spaces here are for asymptotically flat manifolds, whose definition is referred to, e.g. \cite[Section 2]{Eichmair-Huang-Lee-Schoen:2016}.)  Let $n\ge 3$, $s\in (\frac{n-2}{2}, n-2)$ and let $(M, g)$ be an $n$-dimensional asymptotically flat manifold. We define the Banach space of symmetric $(0,2)$-tensors by
\[
	\mathcal{S}=\big\{ h\in \C^{2,\alpha}_{-s}(M): h^\intercal=0 \mbox{ and } DH|_g(h) = 0 \mbox{ on } \Sigma\big\}.
\] 
We recall the formulas for the linearized scalar curvature $DR|_g$ and linearized mean curvature $DH|_g$:
\begin{align*}
DR|_g(h)&=-\Delta_g \tr _g h + \mathrm{div}_g \mathrm{div}_g h -  h \cdot \Ric_g\\
DH|_g(h)&= \tfrac{1}{2} \nu(\tr \, h^\intercal) - \mathrm{div}_\Sigma \omega - \tfrac{1}{2} h(\nu, \nu) H_g,
\end{align*}
where  $\nu$ is the unit normal on $\Sigma$ pointing to infinity, $\omega$ is the one-form on the tangent bundle of $\Sigma$ defined by $\omega(\cdot) = h(\nu, \cdot)$.

 The following theorem  is obtained by Anderson and Jauregui for $n=3$ in \cite[Proposition 2.4]{Anderson:2019tm}
 \begin{manualtheorem}{\ref{thm:AF}}
Let $(M, g)$ be an $n$-dimensional asymptotically flat manifold.  Then the linearized scalar curvature map $DR|_g: \mathcal{S} \to \C^{0,\alpha}_{-2-s}(M)$ is surjective. As a consequence, the map 
\[
h\in  \C^{2,\alpha}_{-s}(M) \longmapsto \big(DR|_{g} (h), h^\intercal, DH|_{g}(h)\big) \in \C^{0,\alpha}_{-2-s}(M) \times \C^{2,\alpha}(\Sigma)\times \C^{1,\alpha}(\Sigma)
\] is surjective. \end{manualtheorem}

The key to prove the above theorem is to show that the map $DR|_g$ has closed range, as stated in the following lemma. An alternative proof of the lemma is given by Zhongshan An  \cite[Section 2.1]{An:2021uv} (stated for $n=3$).  We thank her for explaining her proof to us. It appears that both proofs  introduce an auxiliary scalar function to obtain a one-dimensional higher warped product with $h$ in a ``spacetime.'' Their proofs consider the linearized Einstein tensor on the spacetime (so that the equation for $DR|_g$ appears as the ``time'' component) and use the result of Anderson and Khuri on ellipticity of the system under some subtle gauge conditions~\cite[Lemma 3.2]{Anderson:2013vx}.  Here we provide an elementary and self-contained proof for general $n\ge 3$ that analyzes the operator $DR|_g$ directly. 

\begin{lemma}
Let $(M, g)$ be an asymptotically flat manifold with compact boundary $\Sigma$. Then the map $DR|_g : \mathcal{S}\to \C^{0,\alpha}_{-2-s}(M)$ has  finite-dimensional cokernel, and hence it has closed range.
\end{lemma}
\begin{proof}
 For any scalar function $v$ and  any symmetric $(0,2)$-tensor $h$, we define the differential operator $L$ and the boundary operator $B$ on $(v, h)$ by
\begin{align*}
	\begin{array}{l} L(v, h) =  \left\{ \begin{array}{l} (1-n) \Delta v - v R_g + DR|_g(h) \\ \Delta h \end{array}\right. \quad \mbox{ in }M	\\
B(v, h) = \left\{ \begin{array}{l} v g^\intercal + h^\intercal\\
	-\tfrac{1}{2} v H_g +\tfrac{1}{2} (n-1)\nu(v) + \tfrac{1}{2} \nu(\tr h^\intercal) - \Div_\Sigma \omega- \tfrac{1}{2} h(\nu, \nu) H_g\\
		h(\nu,\nu)\\
(\Div h - d v)^\intercal\end{array}\quad\mbox{ on } \Sigma \right.
		\end{array}.
\end{align*}
We note that the first component in $L(v, h)$ is exactly $DR|_g (vg+ h)$ and the second component in $L$ is the tensor Laplacian on $h$. We also note that the first two components in $B(v, h)$ are exactly the Bartnik boundary data of $vg+h$, while the rest of components in $B(v, h)$ are somewhat arbitrary -- they are chosen so that the fact \eqref{eq:regular} below holds.

We will verify that $L(v, h)$ is elliptic and that the boundary operator $B(v, h)$ satisfies the corresponding complementing boundary condition in the sense of Agmon-Douglis-Nirenberg~\cite{Agmon:1964us}.  Once we verify those, we then have that $L$ is a Fredholm operator on the space of $(v, h)\in \C^{2,\alpha}_{-s}(M)$ subject to the boundary condition $B(v, h)=0$, so $\range L$ is of finite codimension. To complete the proof, we note that $\range DR|_g$ (from the space $\mathcal{S}$) contains the range of the first scalar component of $L$ and hence $\range DR|_g$  is also of finite codimension. 

Let $\{ e_1, \dots, e_n\}$ be an orthonormal local frame so that $e_n = \nu$ along $\Sigma$. Denote the number 
\[
	N = 1+\frac{n(n-1)}{2} + n.
\]
For the purpose to assign an order for the unknowns $(v, h)$, let $\iota$ be a bijection from the index space $\{ (j,k) : 1\le j < k \le n\}$ to the integers $\{ 2, \dots, N\}$.  With respect to the local  frame, we write $(v, h)$ as $(u_1, \dots, u_N)$ where $u_{\iota(j,k)} = h_{jk}$. We express the $i$-th (scalar) differential equation of  $Lu$ and the $h$-th (scalar) differential equation of  $Bu$ respectively by, for $i, h=1,\dots, N$, 
\begin{align*}
	 (Lu)_i &= \sum_{j=1}^{N} \ell_{ij}(\partial) u_j\\
	 (B u)_h &= \sum_{j=1}^{N} B_{hj}(\partial ) u_j.
\end{align*}
To identify the symbols, we substitute\footnote{Recall that the substitution of $\partial$ by $\xi$ to get $\ell_{ij}(\xi), B_{hj}(\xi)$ means that for any multi-index $I = (I_1, \dots, I_n)$, the partial derivative $e_1^{I_1}\dots e_n^{I_n}$ is substituted by the polynomial $\xi_1^{I_1}\dots \xi_n^{I_n}$. } the differential operator $\partial$ in $\ell_{ij}(\partial)$ and $B_{hj}(\partial)$ with the polynomials of  $\xi = (\xi_1, \dots, \xi_n)$. Let $\ell_{ij}'(\xi)$ consist of the homogeneous polynomials  in $\ell_{ij}(\xi)$ of degree $2$.   The \emph{principle symbol} of the differential operator $L$, denoted by $L'(\xi)$,  is the square matrix of size $N$ whose $(i,j)$-th entry  is $\ell'_{ij}(\xi)$. We express the principal symbol $L'(\xi)$ in the following $[1]+[N-1]$ block form:
 \begin{equation} \label{eq:block}
	L'(\xi)=  \left[\begin{array}{c|c} (1-n)|\xi|^2& * \\ 
	\hline 0 & |\xi|^2 I_{(N-1)}  
	\end{array}\right]
\end{equation}
where $|\xi| = \sqrt{\xi_1^2 +\dots+ \xi_n^2}$, $I_{N-1}$ is the identity  matrix of size $(N-1)$, and the asterisk represents the row that contains polynomials in $\xi$ (homogeneous of degree~$2$). Since the $L'(\xi)$ is an upper triangular matrix, we compute 
\[
	\mathrm{det} \, 	L'(\xi)=(1-n) |\xi|^{2N} \neq 0 \quad\mbox{ for any real $\xi\neq 0$}.
\]
Thus, $L$ is an elliptic operator as in \cite[p. 39]{Agmon:1964us}. (Implicitly, we use the weights $s_i=0$, $t_j=2$ for all $i,j$ where the weights $s_i, t_j$ (as well the $r_h$ that appears below) are defined as in \cite{Agmon:1964us}.)

We now identify the principal symbol of the boundary operator $B$. Let $B'_{hj}(\xi)$ consist of the polynomials in $B_{hj}(\xi)$ of the highest degree in each $(Bu)_h$ equation. More precisely, the highest degree in the first and third equations in $B$ is $0$ (which corresponds to the weight $r_h=-2$) and the highest degree in the other equations of $B$ is $1$ (which corresponds the weight $r_h=-1$). The principal symbol, denoted by $B'(\xi)$, is an $N\times N$ matrix whose $(h, j)$-th entry is $B_{hj}'(\xi)$. Instead of writing out $B'(\xi)$ in the matrix form, we list all the rows of the matrix $B'(\xi) \begin{bmatrix} u_1\\ \vdots\\ u_N\end{bmatrix}$ (and substitute $(u_1, \dots, u_N)$ with $(v, h)$): for $\alpha, \beta = 1,\dots, n-1$ and $i=1,\dots, n$:
\begin{align*}
	&v\delta_{\alpha \beta} + h_{\alpha\beta}\\
	&\tfrac{1}{2}(n-1) \xi_n v + \tfrac{1}{2} \xi_n \sum_{\alpha=1}^{n-1} h_{\alpha\alpha} -\sum_{\alpha=1}^{n-1} \xi_\alpha h_{n\alpha} \\
		&h_{n n}\\
	 &\left(\sum_{i=1}^n \xi_i h_{i\alpha} \right)- \xi_\alpha v.
\end{align*}
Before we proceed, we verify the fact to be used later that 
\begin{align}\label{eq:regular}
	\Det B'( \xi)\neq 0  \quad \mbox{for all $\xi=(\xi', \mathsf{i}) $ with $|\xi'|=1$ (and $\xi_n = \mathsf{i}$)}
	\end{align}
where $\xi':=(\xi_1,\dots, \xi_{n-1})$ are real and $\mathsf{i}$ denotes a unit imaginary number. Since $B'(\xi)$ is a square matrix, it is equivalent to directly verifying that the only solution $( v, h)$ to the following linear system is the zero solution:
\begin{align*}
	&v\delta_{\alpha \beta} + h_{\alpha\beta}=0\\
	&\tfrac{1}{2} (n-1) \mathsf{i} v + \tfrac{1}{2}  \mathsf{i}\sum_{\alpha=1}^{n-1} h_{\alpha\alpha} -\sum_{\alpha=1}^{n-1} \xi_\alpha h_{n\alpha}=0 \\
	& h_{n n}=0\\
	&\left(\sum_{\beta=1}^{n-1} \xi_\beta h_{\beta \alpha} \right) + \mathsf{i} \, h_{n\alpha} - \xi_\alpha v=0. 
\end{align*}

For the rest of the proof, we show how \eqref{eq:regular} implies that the boundary operator $B$ satisfies the complementary boundary condition of \cite[pp. 42-43]{Agmon:1964us}. As preparation, we compute the adjoint matrix of $L'(\xi)$:
\begin{align}
\adj L'(\xi)&:= \Det L'(\xi) (L'(\xi))^{-1} =   (1-n)|\xi|^{2N} \left[\begin{array}{c|c}\frac{ 1}{(1-n)}|\xi|^{-2}& \frac{-1}{1-n} |\xi|^{-2}* \notag \\ 
	\hline 0 & |\xi|^{-2} I_{(N-1)}  
	\end{array}\right]\\
	&=(1-n) |\xi|^{2(N-1)} \left[\begin{array}{c|c}\frac{ 1}{(1-n)}& \frac{-1}{1-n}* \\ 
	\hline 0 &  I_{(N-1)}  
	\end{array}\right]\label{eq:adj}
\end{align}
where in the first line $ \frac{-1}{1-n} |\xi|^{-2}*$ denotes $\frac{-1}{1-n} |\xi|^{-2}$ multiplying the asterisk row in \eqref{eq:block}, and similarly for the second line. From now on, we assume $|\xi'|=1$ and $\xi_n=\tau$. Denote the polynomial $M^+(\tau) = (\tau-\mathsf{i})^N$   in $\tau$.  (The polynomial is defined so that it has the root $\mathrm{i}$  with multiplicity $N$ where $\mathrm{i}$ is the root (with multiplicity $N$) of  $\det L'(\xi)=0$ that has positive imaginary part.)

The boundary operator is said to be \emph{complementary} if for any $\xi= (\xi', \tau )$ with $|\xi'|=1$ and for any constants $C_1,\dots, C_N$ satisfying
\begin{align}\label{eq:complementary}
	B'(\xi) \, \adj L'(\xi) \begin{bmatrix} C_1\\\vdots\\ C_N\end{bmatrix} = M^+(\tau) P(\tau)
\end{align}
where $P(\tau)$ is a column vector whose entries are polynomials in $\tau$, we must have $C_1=\dots= C_N=0$.

The expression of the adjoint matrix in \eqref{eq:adj} now takes the form for $|\xi'|=1$: 
\[
	\adj L'(\xi', \tau )= (1-n) (\tau-\mathsf{i})^{N-1}(\tau+\mathsf{i})^{N-1} \left[\begin{array}{c|c}\frac{ 1}{(1-n)}& \frac{-1}{1-n}* \\ 
	\hline 0 &  I_{(N-1)}\end{array}\right].
\]
So we can cancel out the common factor $(\tau-i)^{N-1}$ from the both sides of \eqref{eq:complementary}  and get 
 \[
	(1-n) (\tau+\mathsf{i})^{N-1} B'(\xi',\tau )  \left[\begin{array}{c|c}\frac{ 1}{(1-n)}& \frac{-1}{1-n}* \\ 
	\hline 0 &  I_{(N-1)}\end{array}\right] \begin{bmatrix} C_1\\\vdots\\ C_N\end{bmatrix} =(\tau-\mathsf{i}) P(\tau).
\]
To solve for $C_1,\dots, C_N$, we set $\tau=\mathsf{i}$  to get the homogeneous system:
 \[
	B'(\xi', \mathsf{i})  \left[\begin{array}{c|c}\frac{ 1}{(1-n)}& \frac{-1}{1-n}* \\ 
	\hline 0 &  I_{(N-1)}\end{array}\right] \begin{bmatrix} C_1\\\vdots\\ C_N\end{bmatrix} =0.
\]
We have shown that the first matrix in the left hand side is nonsingular in \eqref{eq:regular} and the matrix in the block form is clearly nonsingular, so we conclude that $C_1=\dots=C_N=0$.

\end{proof}

We explain how to obtain Theorem~\ref{thm:AF} from the above lemma. Once $DR|_g$ has closed range, it suffices to show that the adjoint operator has trivial kernel, i.e. for any $V\in (\C^{0,\alpha}_{-2-s}(M))^*$ solving $DR|_g^* (V)=0$ weakly, we have $V\equiv 0$. Elliptic regularity implies that a weak solution $V$ is at least $\C^{2,\alpha}_{\mathrm{loc}}$ and thus $DR|_g^*(V)=0$ in the pointwise sense; that is, $V$ is a static potential. Recall that a nonzero static potential on an asymptotically flat manifold must be asymptotic to a nontrivial linear combination of $\{1, x_1,\dots, x_n\}$ (see, e.g. \cite[Proposition B.4]{Huang:2018ue}), which, however, cannot be a bounded functional on $\C^{0,\alpha}_{-2-s}(M)$. We then conclude that $V$ is identically zero.

\bibliography{ALH_rigidity}

\end{document}